\documentclass[fleqn,11pt,oneside]{article}

\usepackage{graphicx,hyperref}
\usepackage{bbm}
\usepackage{bm}
\usepackage{amsmath,amsthm,amssymb,amscd}
\usepackage{accents}
\usepackage{subfig}
\usepackage{multirow}
\usepackage{setspace} 
\usepackage[font=doublespacing]{caption}
\usepackage[ top=1.1in, 
             bottom=1.1in, 
             left=1.6in, 
             right=1.1in, 
             driver=dvipdfm, 
             ]{geometry}

\usepackage{tikz}
\usetikzlibrary{math,calc}
\usepackage{pgfplots}

\usepackage{tabu}
\newtabulinestyle{mydashed=on 2.8pt off 2.8pt}

\usepackage{dashbox}
\newcommand\boxit[1]{%
    \par \noindent \fbox{\begin{minipage}{\textwidth}%
    #1\vfill\end{minipage}} }

\renewcommand{\boxit}[1]{#1}

\newtheorem{theorem}{Theorem}[section]
\newtheorem{lemma}[theorem]{Lemma}

\newtheorem{definition}{Definition}[section]
\newtheorem{observe}{Observation}[section]
\newtheorem{remark1}[observe]{Remark}

\newenvironment{observation}{\begin{observe} \rm}{\end{observe}}
\newenvironment{remark}{\begin{remark1} \rm}{\end{remark1}}

\def\qed{\hfill$\blacksquare$\\} \renewenvironment{proof}{\noindent {\bf 
Proof.}}{\qed}

\newcommand{\overbar}[1]{\mkern 1.5mu\overline{\mkern-1.5mu%
    #1\mkern-1.5mu}\mkern 1.5mu}

\newcommand{\sgn}{\mathop{\mathrm{sgn}}}

\newcommand{\inner}[2]{\ensuremath{\langle #1,#2 \rangle}}

\def\R{\mathbbm{R}}
\def\1{\mathbbm{1}}

\title{Zernike Polynomials: Evaluation, Quadrature, and Interpolation}
\author{Philip Greengard, Kirill Serkh}

\begin{document}

%
%

\begin{center}
   \begin{minipage}[t]{6.0in}

Zernike polynomials are a basis of orthogonal polynomials 
on the unit disk that are a natural basis for representing 
smooth functions. They arise in a number of applications
including optics and atmospheric sciences. 
In this paper, we provide a self-contained reference on Zernike 
polynomials, algorithms for evaluating them, and 
what appear to be new numerical schemes for 
quadrature and interpolation. We also introduce new properties of 
Zernike polynomials in higher dimensions.
The quadrature rule and interpolation scheme use a tensor
product of equispaced nodes in the angular direction 
and roots of certain Jacobi polynomials in the radial 
direction. An algorithm for finding the roots of these Jacobi 
polynomials is also described.
The performance of the interpolation and 
quadrature schemes is illustrated through numerical experiments.
Discussions of higher dimensional Zernike polynomials are
included in appendices.

\thispagestyle{empty}

  \vspace{ -100.0in}

  \end{minipage}

\end{center}

\vspace{ 2.60in}

\begin{center}

   \begin{minipage}[t]{4.4in}

\begin{center}

{\bf  Zernike Polynomials: Evaluation, Quadrature, and Interpolation} \\

  \vspace{ 0.50in}

Philip Greengard$^{\dagger\, \star}$ and Kirill Serkh$\mbox{}^{\ddagger \, \diamond}$

\end{center}

  \vspace{ -100.0in}

  \end{minipage}

           \end{center}

 \vspace{ 2.00in}


\vfill

\noindent 
$^{\star}$ This author's work  was supported in
part under ONR N00014-14-1-0797, AFOSR FA9550-16-0175, and NIH 1R01HG008383. \\
\noindent 
$\mbox{}^{\diamond}$  This author's work  was supported in part by the NSF
Mathematical Sciences Postdoctoral Research Fellowship (award no.~1606262)
and AFOSR FA9550-16-1-0175.
\\

\vspace{2mm}
\vspace{2mm}

\noindent
$\mbox{}^{\dagger}$ Dept.~of Mathematics, Yale University, New Haven, CT 06511 \\
\noindent
$\mbox{}^{\ddagger}$ Courant Institute of Mathematical Sciences,
New York University, New York, NY 10012

\vspace{2mm}


\vfill
\eject


\newpage

\section{Introduction}
Zernike polynomials are a family of orthogonal polynomials
that are a natural basis for the approximation of smooth 
functions on the unit disk. Among other applications, 
they are widely used in optics and atmospheric sciences and 
are the natural basis for representing Generalized Prolate 
Spheroidal Functions (see \cite{slepian}).

In this report, we provide a self-contained
reference on Zernike polynomials, including tables of 
properties, an algorithm for their evaluation, 
and what appear to be new numerical schemes
for quadrature and interpolation. 
We also introduce properties
of Zernike polynomials in higher dimensions and
several classes of numerical algorithms for Zernike 
polynomial discretization in $\R^n$.
The quadrature and interpolation 
schemes provided use a tensor product of equispaced nodes 
in the angular direction and roots of certain Jacobi 
polynomials in the radial direction. An 
algorithm for the evaluation of these roots is also introduced.

The structure of this paper is as follows. 
In Section \ref{secmprem}
we introduce several technical lemmas and provide 
basic mathematical background that will be used in subsequent
sections. 
In Section \ref{secnumev} we provide a recurrence relation 
for the evaluation of Zernike polynomials.
Section \ref{seczernquad} describes a scheme for integrating 
Zernike polynomials over the unit disk.
Section \ref{secapprox} contains an algorithm for the interpolation of 
Zernike polynomials.
In Section \ref{secnumres} we give results of numerical experiments 
with the quadrature and interpolation schemes introduced in 
the preceding sections. 
In Appendix A, we describe properies of Zernike polynomials in 
$\R^n$.
Appendix B contains a description of an algorithm for the evaluation of 
Zernike polynomials in $\R^n$.
Appendix C includes an description of Spherical Harmonics in 
higher dimensions. 
In Appendix D, an overview is provided of the family of Jacobi
polynomials whose roots are used in numerical algorithms 
for high-dimensional Zernike polynomial discretization. Appendix
D also includes a description of an algorithm for computing 
their roots. 
Appendix E contains notational conventions for Zernike polynomials. 
\section{Mathematical Preliminaries}\label{secmprem}
In this section, we introduce notation and several technical
lemmas that will be used in subsequent sections.

For notational convenience and ease of generalizing to 
higher dimensions, we will be denoting by $S_N^\ell(\theta): 
\R \rightarrow \R$, the function defined by the formula
\begin{equation}\label{20}
S_N^\ell(\theta) = \left\{
  \begin{array}{ll}
  (2\pi)^{-1/2} & \mbox{if $N = 0$}, \\
  \sin(N\theta)/\sqrt{\pi} & \mbox{if $\ell = 0$, $N>0$}, \\
  \cos(N\theta)/\sqrt{\pi} & \mbox{if $\ell = 1$, $N>0$}.
  \end{array}
\right.
\end{equation}
where $\ell \in\{0,1\}$, and $N$ is a non-negative integer.
In accordance with standard practice, we will denoting by 
$\delta_{i,j}$ the function defined by the formula
  \begin{align}\label{40}
\delta_{i,j} = \left\{
  \begin{array}{ll}
  1 & \mbox{if $i = j$}, \\
  0 & \mbox{if $i \ne j$}.
  \end{array}
\right.
  \end{align}
The following lemma is a classical fact from elementary calculus.
\begin{lemma}\label{60}
For all $n \in \{1,2,...\}$ and for any integer $k\geq n+1$,
\begin{equation}\label{80}
\frac{1}{k}\sum_{i=1}^k \sin(n\theta_i)=\int_0^{2\pi}
\sin(n\theta)d\theta=0
\end{equation}
and 
\begin{equation}\label{100}
\frac{1}{k}\sum_{i=1}^k \cos(n\theta_i)=\int_0^{2\pi}
\cos(n\theta)d\theta=0
\end{equation}
where
\begin{equation}\label{120}
\theta_i=i\frac{2\pi}{k}
\end{equation}
for $i=1,2,...,k$.
\end{lemma}
The following technical lemma will be used in Section 
\ref{seczernquad}.
\begin{lemma}\label{140}
For all $m\in \{0,1,2,...\}$, 
the set of all points $(N,n,\ell) \in \R^3$ such that 
$\ell\in \{0,1\}$, $N,n$ are non-negative integers,
and $N+2n \leq 2m-1$ contains exactly $2m^2+2m$ elements.
\end{lemma}
\begin{proof}
Lemma \ref{140} follows immediately from the 
fact that the set of all pairs of non-negative integers
$(N,n)$ satisfying $N+2n \leq 2m-1$ has $m^2+m$ elements
where $m$ is a non-negative integer.
\end{proof}
The following is a classical fact from elementary functional 
analysis. A proof can be found in, for example, \cite{stoer}.
\begin{lemma}\label{130}
Let $f_1,...,f_{2n-1}:[a,b]\rightarrow \R$ be a set of 
orthonormal functions such that for all $k\in \{1,2,...,2n-1\}$, 
\begin{equation}
\int_a^b f_k(x) dx = \sum_{i=1}^n f_k(x_i) \omega_i dx
\end{equation}
where $x_i \in [a,b]$ and $\omega_i \in \R$. 
Let $\phi:[a,b]\rightarrow \R$ be defined by the formula
\begin{equation}
\phi(x)=a_1f_1(x)+...+a_{n-1}f_{n-1}(x).
\end{equation}
Then,
\begin{equation}
a_k=\int_a^b \phi(x) f_k(x) dx=\sum_{i=1}^n \phi(x_i)f_k(x_i)\omega_i.
\end{equation}
for all $k\in \{1,2,...,n-1\}$.
\end{lemma}
\subsection{Jacobi Polynomials}\label{secjacpol}
In this section, we define Jacobi
polynomials and summarize some of their properties.\\
Jacobi Polynomials, denoted $P_n^{(\alpha,\beta)}$, are orthogonal 
polynomials on the interval $(-1,1)$ with respect to weight function
\begin{equation}
w(x)=(1-x)^{\alpha} (1+x)^{\beta}.
\end{equation}
Specifically, for all non-negative integers $n,m$ with $n\neq m$ and 
real numbers $\alpha,\beta>-1$,
\begin{equation}\label{192}
\int_{-1}^{1}P_n^{(\alpha,\beta)}(x)P_m^{(\alpha,\beta)}(x)
(1-x)^{\alpha} (1+x)^{\beta}dx=0
\end{equation}
The following lemma, provides a stable recurrence relation that can 
be used to evaluate a particular class of Jacobi Polynomials 
(see, for example, \cite{abramowitz}).
\begin{lemma}
For any integer $n\geq1$ and $N\geq0$,
\begin{align}\label{195}
&\hspace*{-5em} P_{n+1}^{(N,0)}(x)=
\frac{(2n+N+1)N^2+(2n+N)(2n+N+1)(2n+N+2)x}
{2(n+1)(n+N+1)(2n+N)}
P_n^{(N,0)}(x) \notag \\
&-\frac{2(n+N)(n)(2n+N+2)}
{2(n+1)(n+N+1)(2n+N)}
P_{n-1}^{(N,0)}(x),
\end{align}
where 
\begin{equation}
P_{0}^{(N,0)}(x)=1
\end{equation}
and 
\begin{equation}
P_{1}^{(N,0)}(x)=\frac{N+(N+2)x}{2}.
\end{equation}
The Jacobi Polynomial $P_n^{(N,0)}$ is defined in (\ref{192}).
\end{lemma}
The following lemma provides a stable recurrence relation that 
can be used to evaluate derivatives of a certain class of 
Jacobi Polynomials. It is readily obtained by differentiating 
(\ref{195}) with respect to $x$, 
\begin{lemma}\label{196}
For any integer $n\geq1$ and $N\geq0$,
\begin{align}\label{197}
&\hspace*{-4em} P_{n+1}^{(N,0)\prime}(x)=
\frac{(2n+N+1)N^2+(2n+N)(2n+N+1)(2n+N+2)x}
{2(n+1)(n+N+1)(2n+N)}
P_n^{(N,0)\prime}(x) \notag \\
&-\frac{2(n+N)(n)(2n+N+2)}
{2(n+1)(n+N+1)(2n+N)}
P_{n-1}^{(N,0)\prime}(x) \notag \\
&+\frac{(2n+N)(2n+N+1)(2n+N+2)}
{2(n+1)(n+N+1)(2n+N)}
P_n^{(N,0)}(x),
\end{align}
where 
\begin{equation}
P_{0}^{(N,0)\prime}(x)=0
\end{equation}
and 
\begin{equation}
P_{1}^{(N,0)\prime}(x)=\frac{(N+2)}{2}.
\end{equation}
The Jacobi Polynomial $P_n^{(N,0)}$ is defined in (\ref{192}) 
and $P_n^{(N,0)\prime}(x)$ denotes the derivative 
of $P_n^{(N,0)}(x)$ with respect to $x$.
\end{lemma}
The following lemma, which provides a differential equation 
for Jacobi polynomials, can be found in \cite{abramowitz}
\begin{lemma}\label{205}
For any integer $n$,
\begin{equation}\label{210}
(1-x^2)P_n^{(k,0)\prime\prime}(x)+(-k-(k+2)x)P_n^{(k,0)\prime}(x)
+n(n+k+1)P_n^{(k,0)}(x)=0
\end{equation}
for all $x\in [0,1]$ where $P_n^{(N,0)}$ is defined in (\ref{192}).
\end{lemma}
\begin{remark}
We will be denoting by $\widetilde{P}_n: [0,1] \rightarrow \R$ the
shifted Jacobi polynomial defined for any non-negative integer $n$
by the formula
\begin{equation}\label{215}
\widetilde{P}_n(x)=\sqrt{2n+2}P_n^{(1,0)}(1-2x)
\end{equation}
where $P_n^{(1,0)}$ is defined in (\ref{192}). The roots of 
$\widetilde{P}_n$ will be used in Section \ref{seczernquad} and Section 
\ref{secapprox} in the design of quadrature and interpolation
schemes for Zernike polynomials. 
\end{remark}
It follows immediately from the combination of 
(\ref{192}) and (\ref{215}) that the polynomials $\widetilde{P}_n$ 
are orthogonal on $[0,1]$ with respect to weight function
\begin{equation}
w(x)=x. 
\end{equation}
That is, for any non-negative integers $i,j$, 
\begin{equation}\label{740}
\int_0^1\widetilde{P}_i(r)\widetilde{P}_j(r)rdr=\delta_{i,j}.
\end{equation}
\subsection{Gaussian Quadratures}\label{secgengauss}
In this section, we introduce Gaussian Quadratures.
\begin{definition}\label{220}
A Gaussian Quadrature with respect to a set of functions
$f_1,...,f_{2n-1}:[a,b]\rightarrow \mathbb{R}$ and non-negative weight 
function $w:[a,b]\rightarrow \mathbb{R}$ is a set of $n$ nodes, 
$x_1,...,x_n\in[a,b]$, and $n$ weights,
$\omega_1,...,\omega_n \in \R$, such that, for any integer $j\leq 2n-1$,
\begin{equation}\label{240}
\int_a^b f_j(x)w(x)dx=\sum_{i=0}^n\omega_i f_j(x_i).
\end{equation}
\end{definition}
The following is a well-known lemma from numerical analysis.
A proof can be found in, for example, \cite{stoer}.
\begin{lemma}\label{260}
Suppose that $p_0,p_1,...:[a,b]\rightarrow \R$ 
is a set of orthonormal polynomials with respect 
to some non-negative weight function $w:[a,b]
\rightarrow \mathbb{R}$ such that polynomial $p_i$ is of 
degree $i$. Then, \\\\
i) Polynomial $p_i$ has exactly $i$ roots on $[a,b]$.\\\\
ii) For any non-negative integer $n$ and for $i=0,1,...,2n-1$, we have
\begin{equation}\label{280}
\int_a^b p_i(x)w(x)dx=\sum_{k=1}^n\omega_k p_i(x_k)
\end{equation}
where $x_1,...,x_n \in [a,b]$ are the $n$ roots
of $p_n$ and where weights $\omega_1,...,\omega_n \in \R$ solve 
the $n \times n$ system of linear equations
\begin{equation}\label{300}
\sum_{k=1}^n \omega_k p_j(x_k) = \int_a^b w(x)p_j(x)dx
\end{equation}
with $j=0,1,...,n-1$. \\\\
iii) The weights, $\omega_i$, satisfy the identity,
\begin{equation}\label{320}
\omega_i=\left(\sum_{k=0}^{n-1} p_k(x_i)^2 \right)^{-1}
\end{equation}
for $i=1,2,...,n$.
\end{lemma}
\subsection{Zernike Polynomials} \label{seczern}
In this section, we define Zernike Polynomials 
and describe some of their basic properties. 

Zernike polynomials are a family of orthogonal polynomials 
defined on the unit ball in $\R^n$. In this paper, we primarily
discuss Zernike polynomials in $\R^2$, however nearly all of the 
theory and numerical machinery in two dimensions generalizes
naturally to higher dimensions. The mathematical properties
of Zernike polynomials in $\R^n$ are included in Appendix A.

Zernike Polynomials are defined via the formula
\begin{equation}\label{340}
Z_{N,n}^\ell(x) = R_{N,n}(r)S_N^{\ell}(\theta)
\end{equation}
for all $x\in \R^2$ such that $\|x\|\leq1$, $(r,\theta)$ is the 
representation of $x$ in polar coordinates, $N,n$ are non-
negative integers, $S_N^\ell$ is defined in (\ref{20}), 
and  $R_{N,n}$ are polynomials of degree $N+2n$ defined by the
formula
  \begin{align}\label{360}
R_{N,n}(x) = x^N \sum_{k=0}^n (-1)^k {n+N+\frac{p}{2} \choose k} 
  {n\choose k} (x^2)^{n-k} (1-x^2)^k,
   \end{align}
for all $0\le x\le 1$. Furthermore, for any non-negative 
integers $N,n,m$,
  \begin{align}\label{1.45}
\int_0^1 R_{N,n}(x)R_{N,m}(x) x\, dx = \frac{\delta_{n,m}}
  {2(2n+N+1)}
  \end{align}
and
  \begin{align}\label{1.35}
R_{N,n}(1) = 1.
  \end{align}
We define the normalized polynomials $\overline{R}_{N,n}$ via the formula
  \begin{align}\label{1.65}
\overline{R}_{N,n}(x) = \sqrt{2(2n+N+1)} R_{N,n}(x),
  \end{align}
so that
  \begin{align}\label{1.75}
\int_0^1 \bigl(\overline{R}_{N,n}(x)\bigr)^2 x\, dx = 1,
  \end{align}
where $N$ and $n$ are non-negative integers.
We define the normalized Zernike polynomial, 
$\overline{Z}_{N,n}^\ell$, by the formula
\begin{equation}\label{10.85}
\overline{Z}_{N,n}(x)=\overline{R}_{N,n}(r)S_{N}^\ell(\theta)
\end{equation}
where $x\in\R^2$ satisfies $\|x\|\leq 1$, and $N,n$ are 
non-negative integers. We observe that $\overline{Z}_{N,n}^\ell$
has $L^2$ norm of $1$ on the unit disk. 

In an abuse of notation, we use $Z_{N,n}^\ell(x)$ and 
$Z_{N,n}^\ell(r,\theta)$ interchangeably where $(r,\theta)$
is the polar coordinate representation of $x\in \R^2$.
\section{Numerical Evaluation of Zernike Polynomials}\label{secnumev}
In this section, we provide a stable recurrence relation 
(see Lemma \ref{lem755}) that can be used to 
evaluate Zernike Polynomials.

\begin{lemma}\label{lem755}
The polynomials $R_{N,n}$, defined in (\ref{360}) satisfy the recurrence relation
\begin{align}\label{7550}
&\hspace*{-4em}
R_{N,n+1}(x)= \notag \\
&\hspace*{-4em}-\frac{((2n+N+1)N^2+(2n+N)(2n+N+1)(2n+N+2)(1-2x^2))}
{2(n+1)(n+N+1)(2n+N)}
R_{N,n}(x)  \notag \\
&\hspace*{-4em}-\frac{2(n+N)(n)(2n+N+2)}
{2(n+1)(n+N+1)(2n+N)}
R_{N,n-1}(x)
\end{align}
where $0\le x\le 1$, $N$ is a non-negative integer, $n$ is a 
positive integer, and
\begin{equation}
R_{N,0}(x)=x^N
\end{equation}
and 
\begin{equation}
R_{N,1}(x)=-x^N\frac{N+(N+2)(1-2x^2)}{2}.
\end{equation}
\end{lemma}
\begin{proof}
According to \cite{abramowitz}, for any non-negative integers $n$ and $N$,
\begin{equation}    \label{6.245}
R_{N,n}(x) = (-1)^n x^N P_n^{(N,0)}(1-2x^2),
\end{equation}
where $0\le x\le 1$, $N$ and $n$ are nonnegative integers, and
$P^{(N,0)}_n$ denotes a Jacobi polynomial (see (\ref{192})).

Identity (\ref{7550}) follows immediately from the 
combination of (\ref{6.245}) and (\ref{195}).
\end{proof}

\begin{remark}
The algorithm for evaluating Zernike polynomials using the recurrence
relation in Lemma~\ref{lem755} is known as Kintner's method
(see~\cite{kintner} and, for example,~\cite{chong}).

\end{remark}
\section{Quadrature for Zernike Polynomials}\label{seczernquad}
In this section, we provide a quadrature rule for Zernike Polynomials. 

The following lemma follows immediately from applying Lemma \ref{260} 
to the polynomials $\widetilde{P}_n$ defined in (\ref{215}).
\begin{lemma}\label{lem4.30}
Let $\{r_1,...,r_{m}\}$ be the $m$ roots of $\widetilde{P}_m$ 
(see (\ref{215})) and $\{\omega_1,...,\omega_m\}$
the $m$ weights of the Gaussian quadrature (see (\ref{240}))
for the polynomials $\widetilde{P}_0,\widetilde{P}_1,...,
\widetilde{P}_{2m-1}$ (see (\ref{215})). Then, for any polynomial 
$q$ of degree at most $2m-1$,
\begin{equation}\label{10.60}
\int_0^1 q(x)xdx=\sum_{i=1}^{m}
q(r_i)\omega_i.
\end{equation}
\end{lemma}
The following theorem provides a quadrature rule for Zernike Polynomials.
\begin{theorem}\label{lem4.70}
Let $\{r_1,...,r_{m}\}$ be the $m$ roots of $\widetilde{P}_m$ 
(see (\ref{215})) and $\{\omega_1,...,\omega_m\}$
the $m$ weights of the Gaussian quadrature (see (\ref{240}))
for the polynomials $\widetilde{P}_0,\widetilde{P}_1,...,
\widetilde{P}_{2m-2}$ (see (\ref{215})). Then, for all $\ell 
\in \{0,1\}$ and for all $N,n\in \{0,1,...\}$ such that $N+2n 
\leq 2m-1$,
\begin{equation}\label{10.7}
\int_{D}Z_{N,n}^\ell(x)dx=\sum_{i=1}^{m}
R_{N,n}(r_i)\omega_i
\sum_{j=1}^{2m}
\frac{2\pi}{2m}S_N^\ell(\theta_j)
\end{equation}
where $R_{N,n}$ is defined in (\ref{360}), $\theta_j$ is defined 
by the formula
\begin{equation}
\theta_j=j\frac{2\pi}{2m}
\end{equation}
for $j\in \{1,2,...,2m\}$, and $D\subseteq \R^2$ denotes the unit disk.
Furthermore, there are exactly $2m^2+m$ Zernike Polynomials of 
degree at most $2m-1$.
\end{theorem}
\begin{proof}
Applying a change of variables, 
\begin{equation}\label{10.10}
\int_{D}Z_{N,n}^\ell(x)dx=\int_0^1\int_0^{2\pi}
R_{N,n}(r)S_N^\ell(\theta)rdrd\theta,
\end{equation}
where $Z_{N,n}^\ell$ is a Zernike polynomial (see (\ref{340})) and
where $R_{N,n}$ is defined in (\ref{1.45}).
Changing the order of integration of (\ref{10.10}), we obtain
\begin{equation}\label{10.20}
\int_{D}Z_{N,n}^\ell(x)dx
=\int_0^1rR_{N,n}(r)dr\int_0^{2\pi}S_N^\ell(\theta)d\theta.
\end{equation}
Applying Lemma \ref{60} and Lemma \ref{lem4.30} to (\ref{10.20}),
we obtain
%
\begin{equation}\label{10.30}
\int_{D}Z_{N,n}^\ell(x)dx
=\sum_{i=1}^{m} R_{N,n}(r_i)\omega_i
\sum_{j=1}^{2m} \frac{2\pi}{2m}S_N^\ell(\theta_j)
\end{equation}
for $N+2n\leq 2m-1$. 
The fact that there are exactly $2m^2+m$ Zernike polynomials 
of degree at most $2m-1$ follows immediately 
from the combination of Lemma \ref{140} with the fact that 
there are exactly $m$ Zernike polynomials of degree at most 
$2m-1$ that are of the form $Z_{0,n}^{\ell}$.
\end{proof}
\begin{remark}
It follows immediately from Lemma \ref{lem4.70}
that for all $m\in \{1,2,...\}$, placing $m$
nodes in the radial direction and $2m$ nodes in the 
angular direction (as described in Lemma \ref{lem4.70}),
integrates exactly the $2m^2+m$ Zernike polynomials on 
the disk of degree at most $2m-1$.
\end{remark}
\begin{remark}
The $n$ roots of $\widetilde{P}_n$ (see \ref{740}) can be found by 
using, for example, the algorithm described in Section \ref{secalg}.
\end{remark}
\begin{remark}
For Zernike polynomial discretization in $\R^{k+1}$,
roots of the polynomials $\widetilde{P}_n^k$ are used, where 
$\widetilde{P}_n^k$ is defined by the formula
\begin{equation}\label{210.60}
\widetilde{P}_n^k(x)=\sqrt{k+2n+1}P_n^{(k,0)}(1-2x).
\end{equation}
Properties of this class of Jacobi polynomials 
are provided in Appendix D in addition to an algorithm 
for finding their roots. 
\end{remark}
The following remark illustrates that the advantage
of quadrature rule (\ref{10.7}) is especially noticeable
in higher dimensions. 
\begin{remark}
Quadrature rule (\ref{10.7}) integrates all Zernike polynomials
up to order $2m-1$ using the $m$ roots of $\widetilde{P}_m$ 
(see (\ref{740})) as nodes in the radial direction. 
Using Guass-Legendre nodes instead of roots of $\widetilde{P}_m$
would require using $m+1$ nodes in the radial direction.

The high-dimensional equivalent of quadrature rule (\ref{10.7})
uses the roots of $\widetilde{P}_m^{p+1}$ (see (\ref{760.2})) as
nodes in the radial direction. Using Gauss-Legendre nodes instead 
of these nodes would require using an extra $p+1$ nodes in the 
radial direction or approximately $(p+1)m^{p+1}$ extra nodes total.
\end{remark}
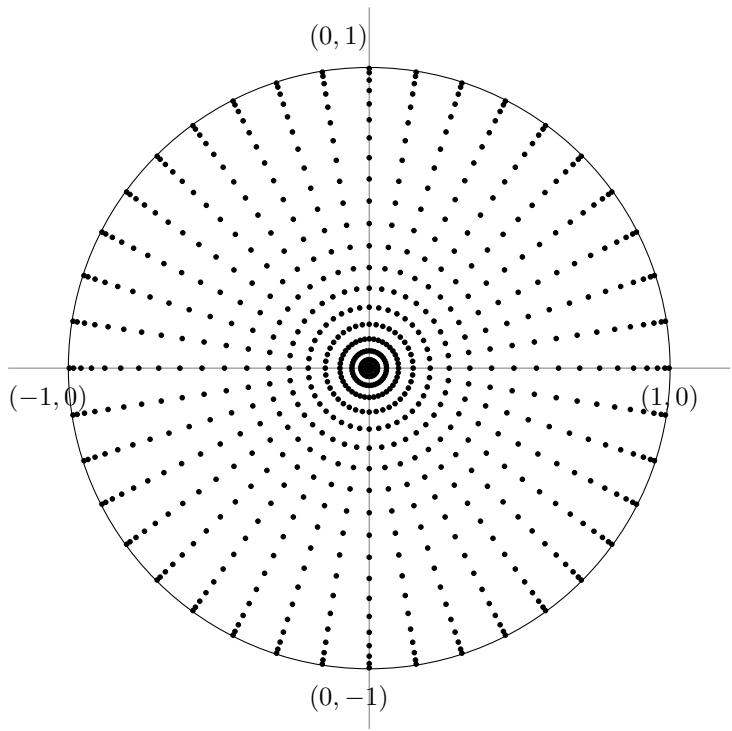
\begin{figure}[h!]
\centering

 \begin{tikzpicture} [scale=4.0]

\draw[black] (0,0) circle (1cm);

\node[text width=0.4cm] at (-1.15,-0.1) {\small $(-1,0)$};
\node[text width=0.4cm] at (-0.15,1.1) {\small $(0,1)$};
\node[text width=0.4cm] at (0.95,-0.1) {\small $(1,0)$};
\node[text width=0.4cm] at (-0.15,-1.1) {\small $(0,-1)$};

\draw[black,opacity=0.5] (-1.2,0) -- (1.2,0);
\draw[black,opacity=0.5] (0,-1.2) -- (0,1.2);

\draw[fill] (  0.00830,  0.00000) circle [radius=0.0075];
\draw[fill] (  0.00820,  0.00130) circle [radius=0.0075];
\draw[fill] (  0.00789,  0.00256) circle [radius=0.0075];
\draw[fill] (  0.00740,  0.00377) circle [radius=0.0075];
\draw[fill] (  0.00671,  0.00488) circle [radius=0.0075];
\draw[fill] (  0.00587,  0.00587) circle [radius=0.0075];
\draw[fill] (  0.00488,  0.00671) circle [radius=0.0075];
\draw[fill] (  0.00377,  0.00740) circle [radius=0.0075];
\draw[fill] (  0.00256,  0.00789) circle [radius=0.0075];
\draw[fill] (  0.00130,  0.00820) circle [radius=0.0075];
\draw[fill] (  0.00000,  0.00830) circle [radius=0.0075];
\draw[fill] ( -0.00130,  0.00820) circle [radius=0.0075];
\draw[fill] ( -0.00256,  0.00789) circle [radius=0.0075];
\draw[fill] ( -0.00377,  0.00740) circle [radius=0.0075];
\draw[fill] ( -0.00488,  0.00671) circle [radius=0.0075];
\draw[fill] ( -0.00587,  0.00587) circle [radius=0.0075];
\draw[fill] ( -0.00671,  0.00488) circle [radius=0.0075];
\draw[fill] ( -0.00740,  0.00377) circle [radius=0.0075];
\draw[fill] ( -0.00789,  0.00256) circle [radius=0.0075];
\draw[fill] ( -0.00820,  0.00130) circle [radius=0.0075];
\draw[fill] ( -0.00830,  0.00000) circle [radius=0.0075];
\draw[fill] ( -0.00820, -0.00130) circle [radius=0.0075];
\draw[fill] ( -0.00789, -0.00256) circle [radius=0.0075];
\draw[fill] ( -0.00740, -0.00377) circle [radius=0.0075];
\draw[fill] ( -0.00671, -0.00488) circle [radius=0.0075];
\draw[fill] ( -0.00587, -0.00587) circle [radius=0.0075];
\draw[fill] ( -0.00488, -0.00671) circle [radius=0.0075];
\draw[fill] ( -0.00377, -0.00740) circle [radius=0.0075];
\draw[fill] ( -0.00256, -0.00789) circle [radius=0.0075];
\draw[fill] ( -0.00130, -0.00820) circle [radius=0.0075];
\draw[fill] (  0.00000, -0.00830) circle [radius=0.0075];
\draw[fill] (  0.00130, -0.00820) circle [radius=0.0075];
\draw[fill] (  0.00256, -0.00789) circle [radius=0.0075];
\draw[fill] (  0.00377, -0.00740) circle [radius=0.0075];
\draw[fill] (  0.00488, -0.00671) circle [radius=0.0075];
\draw[fill] (  0.00587, -0.00587) circle [radius=0.0075];
\draw[fill] (  0.00671, -0.00488) circle [radius=0.0075];
\draw[fill] (  0.00740, -0.00377) circle [radius=0.0075];
\draw[fill] (  0.00789, -0.00256) circle [radius=0.0075];
\draw[fill] (  0.00820, -0.00130) circle [radius=0.0075];
\draw[fill] (  0.02764,  0.00000) circle [radius=0.0075];
\draw[fill] (  0.02730,  0.00432) circle [radius=0.0075];
\draw[fill] (  0.02629,  0.00854) circle [radius=0.0075];
\draw[fill] (  0.02463,  0.01255) circle [radius=0.0075];
\draw[fill] (  0.02236,  0.01625) circle [radius=0.0075];
\draw[fill] (  0.01955,  0.01955) circle [radius=0.0075];
\draw[fill] (  0.01625,  0.02236) circle [radius=0.0075];
\draw[fill] (  0.01255,  0.02463) circle [radius=0.0075];
\draw[fill] (  0.00854,  0.02629) circle [radius=0.0075];
\draw[fill] (  0.00432,  0.02730) circle [radius=0.0075];
\draw[fill] (  0.00000,  0.02764) circle [radius=0.0075];
\draw[fill] ( -0.00432,  0.02730) circle [radius=0.0075];
\draw[fill] ( -0.00854,  0.02629) circle [radius=0.0075];
\draw[fill] ( -0.01255,  0.02463) circle [radius=0.0075];
\draw[fill] ( -0.01625,  0.02236) circle [radius=0.0075];
\draw[fill] ( -0.01955,  0.01955) circle [radius=0.0075];
\draw[fill] ( -0.02236,  0.01625) circle [radius=0.0075];
\draw[fill] ( -0.02463,  0.01255) circle [radius=0.0075];
\draw[fill] ( -0.02629,  0.00854) circle [radius=0.0075];
\draw[fill] ( -0.02730,  0.00432) circle [radius=0.0075];
\draw[fill] ( -0.02764,  0.00000) circle [radius=0.0075];
\draw[fill] ( -0.02730, -0.00432) circle [radius=0.0075];
\draw[fill] ( -0.02629, -0.00854) circle [radius=0.0075];
\draw[fill] ( -0.02463, -0.01255) circle [radius=0.0075];
\draw[fill] ( -0.02236, -0.01625) circle [radius=0.0075];
\draw[fill] ( -0.01955, -0.01955) circle [radius=0.0075];
\draw[fill] ( -0.01625, -0.02236) circle [radius=0.0075];
\draw[fill] ( -0.01255, -0.02463) circle [radius=0.0075];
\draw[fill] ( -0.00854, -0.02629) circle [radius=0.0075];
\draw[fill] ( -0.00432, -0.02730) circle [radius=0.0075];
\draw[fill] (  0.00000, -0.02764) circle [radius=0.0075];
\draw[fill] (  0.00432, -0.02730) circle [radius=0.0075];
\draw[fill] (  0.00854, -0.02629) circle [radius=0.0075];
\draw[fill] (  0.01255, -0.02463) circle [radius=0.0075];
\draw[fill] (  0.01625, -0.02236) circle [radius=0.0075];
\draw[fill] (  0.01955, -0.01955) circle [radius=0.0075];
\draw[fill] (  0.02236, -0.01625) circle [radius=0.0075];
\draw[fill] (  0.02463, -0.01255) circle [radius=0.0075];
\draw[fill] (  0.02629, -0.00854) circle [radius=0.0075];
\draw[fill] (  0.02730, -0.00432) circle [radius=0.0075];
\draw[fill] (  0.05753,  0.00000) circle [radius=0.0075];
\draw[fill] (  0.05683,  0.00900) circle [radius=0.0075];
\draw[fill] (  0.05472,  0.01778) circle [radius=0.0075];
\draw[fill] (  0.05126,  0.02612) circle [radius=0.0075];
\draw[fill] (  0.04655,  0.03382) circle [radius=0.0075];
\draw[fill] (  0.04068,  0.04068) circle [radius=0.0075];
\draw[fill] (  0.03382,  0.04655) circle [radius=0.0075];
\draw[fill] (  0.02612,  0.05126) circle [radius=0.0075];
\draw[fill] (  0.01778,  0.05472) circle [radius=0.0075];
\draw[fill] (  0.00900,  0.05683) circle [radius=0.0075];
\draw[fill] (  0.00000,  0.05753) circle [radius=0.0075];
\draw[fill] ( -0.00900,  0.05683) circle [radius=0.0075];
\draw[fill] ( -0.01778,  0.05472) circle [radius=0.0075];
\draw[fill] ( -0.02612,  0.05126) circle [radius=0.0075];
\draw[fill] ( -0.03382,  0.04655) circle [radius=0.0075];
\draw[fill] ( -0.04068,  0.04068) circle [radius=0.0075];
\draw[fill] ( -0.04655,  0.03382) circle [radius=0.0075];
\draw[fill] ( -0.05126,  0.02612) circle [radius=0.0075];
\draw[fill] ( -0.05472,  0.01778) circle [radius=0.0075];
\draw[fill] ( -0.05683,  0.00900) circle [radius=0.0075];
\draw[fill] ( -0.05753,  0.00000) circle [radius=0.0075];
\draw[fill] ( -0.05683, -0.00900) circle [radius=0.0075];
\draw[fill] ( -0.05472, -0.01778) circle [radius=0.0075];
\draw[fill] ( -0.05126, -0.02612) circle [radius=0.0075];
\draw[fill] ( -0.04655, -0.03382) circle [radius=0.0075];
\draw[fill] ( -0.04068, -0.04068) circle [radius=0.0075];
\draw[fill] ( -0.03382, -0.04655) circle [radius=0.0075];
\draw[fill] ( -0.02612, -0.05126) circle [radius=0.0075];
\draw[fill] ( -0.01778, -0.05472) circle [radius=0.0075];
\draw[fill] ( -0.00900, -0.05683) circle [radius=0.0075];
\draw[fill] (  0.00000, -0.05753) circle [radius=0.0075];
\draw[fill] (  0.00900, -0.05683) circle [radius=0.0075];
\draw[fill] (  0.01778, -0.05472) circle [radius=0.0075];
\draw[fill] (  0.02612, -0.05126) circle [radius=0.0075];
\draw[fill] (  0.03382, -0.04655) circle [radius=0.0075];
\draw[fill] (  0.04068, -0.04068) circle [radius=0.0075];
\draw[fill] (  0.04655, -0.03382) circle [radius=0.0075];
\draw[fill] (  0.05126, -0.02612) circle [radius=0.0075];
\draw[fill] (  0.05472, -0.01778) circle [radius=0.0075];
\draw[fill] (  0.05683, -0.00900) circle [radius=0.0075];
\draw[fill] (  0.09730,  0.00000) circle [radius=0.0075];
\draw[fill] (  0.09611,  0.01522) circle [radius=0.0075];
\draw[fill] (  0.09254,  0.03007) circle [radius=0.0075];
\draw[fill] (  0.08670,  0.04418) circle [radius=0.0075];
\draw[fill] (  0.07872,  0.05719) circle [radius=0.0075];
\draw[fill] (  0.06880,  0.06880) circle [radius=0.0075];
\draw[fill] (  0.05719,  0.07872) circle [radius=0.0075];
\draw[fill] (  0.04418,  0.08670) circle [radius=0.0075];
\draw[fill] (  0.03007,  0.09254) circle [radius=0.0075];
\draw[fill] (  0.01522,  0.09611) circle [radius=0.0075];
\draw[fill] (  0.00000,  0.09730) circle [radius=0.0075];
\draw[fill] ( -0.01522,  0.09611) circle [radius=0.0075];
\draw[fill] ( -0.03007,  0.09254) circle [radius=0.0075];
\draw[fill] ( -0.04418,  0.08670) circle [radius=0.0075];
\draw[fill] ( -0.05719,  0.07872) circle [radius=0.0075];
\draw[fill] ( -0.06880,  0.06880) circle [radius=0.0075];
\draw[fill] ( -0.07872,  0.05719) circle [radius=0.0075];
\draw[fill] ( -0.08670,  0.04418) circle [radius=0.0075];
\draw[fill] ( -0.09254,  0.03007) circle [radius=0.0075];
\draw[fill] ( -0.09611,  0.01522) circle [radius=0.0075];
\draw[fill] ( -0.09730,  0.00000) circle [radius=0.0075];
\draw[fill] ( -0.09611, -0.01522) circle [radius=0.0075];
\draw[fill] ( -0.09254, -0.03007) circle [radius=0.0075];
\draw[fill] ( -0.08670, -0.04418) circle [radius=0.0075];
\draw[fill] ( -0.07872, -0.05719) circle [radius=0.0075];
\draw[fill] ( -0.06880, -0.06880) circle [radius=0.0075];
\draw[fill] ( -0.05719, -0.07872) circle [radius=0.0075];
\draw[fill] ( -0.04418, -0.08670) circle [radius=0.0075];
\draw[fill] ( -0.03007, -0.09254) circle [radius=0.0075];
\draw[fill] ( -0.01522, -0.09611) circle [radius=0.0075];
\draw[fill] (  0.00000, -0.09730) circle [radius=0.0075];
\draw[fill] (  0.01522, -0.09611) circle [radius=0.0075];
\draw[fill] (  0.03007, -0.09254) circle [radius=0.0075];
\draw[fill] (  0.04418, -0.08670) circle [radius=0.0075];
\draw[fill] (  0.05719, -0.07872) circle [radius=0.0075];
\draw[fill] (  0.06880, -0.06880) circle [radius=0.0075];
\draw[fill] (  0.07872, -0.05719) circle [radius=0.0075];
\draw[fill] (  0.08670, -0.04418) circle [radius=0.0075];
\draw[fill] (  0.09254, -0.03007) circle [radius=0.0075];
\draw[fill] (  0.09611, -0.01522) circle [radius=0.0075];
\draw[fill] (  0.14606,  0.00000) circle [radius=0.0075];
\draw[fill] (  0.14426,  0.02285) circle [radius=0.0075];
\draw[fill] (  0.13891,  0.04514) circle [radius=0.0075];
\draw[fill] (  0.13014,  0.06631) circle [radius=0.0075];
\draw[fill] (  0.11817,  0.08585) circle [radius=0.0075];
\draw[fill] (  0.10328,  0.10328) circle [radius=0.0075];
\draw[fill] (  0.08585,  0.11817) circle [radius=0.0075];
\draw[fill] (  0.06631,  0.13014) circle [radius=0.0075];
\draw[fill] (  0.04514,  0.13891) circle [radius=0.0075];
\draw[fill] (  0.02285,  0.14426) circle [radius=0.0075];
\draw[fill] (  0.00000,  0.14606) circle [radius=0.0075];
\draw[fill] ( -0.02285,  0.14426) circle [radius=0.0075];
\draw[fill] ( -0.04514,  0.13891) circle [radius=0.0075];
\draw[fill] ( -0.06631,  0.13014) circle [radius=0.0075];
\draw[fill] ( -0.08585,  0.11817) circle [radius=0.0075];
\draw[fill] ( -0.10328,  0.10328) circle [radius=0.0075];
\draw[fill] ( -0.11817,  0.08585) circle [radius=0.0075];
\draw[fill] ( -0.13014,  0.06631) circle [radius=0.0075];
\draw[fill] ( -0.13891,  0.04514) circle [radius=0.0075];
\draw[fill] ( -0.14426,  0.02285) circle [radius=0.0075];
\draw[fill] ( -0.14606,  0.00000) circle [radius=0.0075];
\draw[fill] ( -0.14426, -0.02285) circle [radius=0.0075];
\draw[fill] ( -0.13891, -0.04514) circle [radius=0.0075];
\draw[fill] ( -0.13014, -0.06631) circle [radius=0.0075];
\draw[fill] ( -0.11817, -0.08585) circle [radius=0.0075];
\draw[fill] ( -0.10328, -0.10328) circle [radius=0.0075];
\draw[fill] ( -0.08585, -0.11817) circle [radius=0.0075];
\draw[fill] ( -0.06631, -0.13014) circle [radius=0.0075];
\draw[fill] ( -0.04514, -0.13891) circle [radius=0.0075];
\draw[fill] ( -0.02285, -0.14426) circle [radius=0.0075];
\draw[fill] (  0.00000, -0.14606) circle [radius=0.0075];
\draw[fill] (  0.02285, -0.14426) circle [radius=0.0075];
\draw[fill] (  0.04514, -0.13891) circle [radius=0.0075];
\draw[fill] (  0.06631, -0.13014) circle [radius=0.0075];
\draw[fill] (  0.08585, -0.11817) circle [radius=0.0075];
\draw[fill] (  0.10328, -0.10328) circle [radius=0.0075];
\draw[fill] (  0.11817, -0.08585) circle [radius=0.0075];
\draw[fill] (  0.13014, -0.06631) circle [radius=0.0075];
\draw[fill] (  0.13891, -0.04514) circle [radius=0.0075];
\draw[fill] (  0.14426, -0.02285) circle [radius=0.0075];
\draw[fill] (  0.20272,  0.00000) circle [radius=0.0075];
\draw[fill] (  0.20023,  0.03171) circle [radius=0.0075];
\draw[fill] (  0.19280,  0.06264) circle [radius=0.0075];
\draw[fill] (  0.18063,  0.09203) circle [radius=0.0075];
\draw[fill] (  0.16401,  0.11916) circle [radius=0.0075];
\draw[fill] (  0.14335,  0.14335) circle [radius=0.0075];
\draw[fill] (  0.11916,  0.16401) circle [radius=0.0075];
\draw[fill] (  0.09203,  0.18063) circle [radius=0.0075];
\draw[fill] (  0.06264,  0.19280) circle [radius=0.0075];
\draw[fill] (  0.03171,  0.20023) circle [radius=0.0075];
\draw[fill] (  0.00000,  0.20272) circle [radius=0.0075];
\draw[fill] ( -0.03171,  0.20023) circle [radius=0.0075];
\draw[fill] ( -0.06264,  0.19280) circle [radius=0.0075];
\draw[fill] ( -0.09203,  0.18063) circle [radius=0.0075];
\draw[fill] ( -0.11916,  0.16401) circle [radius=0.0075];
\draw[fill] ( -0.14335,  0.14335) circle [radius=0.0075];
\draw[fill] ( -0.16401,  0.11916) circle [radius=0.0075];
\draw[fill] ( -0.18063,  0.09203) circle [radius=0.0075];
\draw[fill] ( -0.19280,  0.06264) circle [radius=0.0075];
\draw[fill] ( -0.20023,  0.03171) circle [radius=0.0075];
\draw[fill] ( -0.20272,  0.00000) circle [radius=0.0075];
\draw[fill] ( -0.20023, -0.03171) circle [radius=0.0075];
\draw[fill] ( -0.19280, -0.06264) circle [radius=0.0075];
\draw[fill] ( -0.18063, -0.09203) circle [radius=0.0075];
\draw[fill] ( -0.16401, -0.11916) circle [radius=0.0075];
\draw[fill] ( -0.14335, -0.14335) circle [radius=0.0075];
\draw[fill] ( -0.11916, -0.16401) circle [radius=0.0075];
\draw[fill] ( -0.09203, -0.18063) circle [radius=0.0075];
\draw[fill] ( -0.06264, -0.19280) circle [radius=0.0075];
\draw[fill] ( -0.03171, -0.20023) circle [radius=0.0075];
\draw[fill] (  0.00000, -0.20272) circle [radius=0.0075];
\draw[fill] (  0.03171, -0.20023) circle [radius=0.0075];
\draw[fill] (  0.06264, -0.19280) circle [radius=0.0075];
\draw[fill] (  0.09203, -0.18063) circle [radius=0.0075];
\draw[fill] (  0.11916, -0.16401) circle [radius=0.0075];
\draw[fill] (  0.14335, -0.14335) circle [radius=0.0075];
\draw[fill] (  0.16401, -0.11916) circle [radius=0.0075];
\draw[fill] (  0.18063, -0.09203) circle [radius=0.0075];
\draw[fill] (  0.19280, -0.06264) circle [radius=0.0075];
\draw[fill] (  0.20023, -0.03171) circle [radius=0.0075];
\draw[fill] (  0.26602,  0.00000) circle [radius=0.0075];
\draw[fill] (  0.26274,  0.04161) circle [radius=0.0075];
\draw[fill] (  0.25300,  0.08220) circle [radius=0.0075];
\draw[fill] (  0.23702,  0.12077) circle [radius=0.0075];
\draw[fill] (  0.21521,  0.15636) circle [radius=0.0075];
\draw[fill] (  0.18810,  0.18810) circle [radius=0.0075];
\draw[fill] (  0.15636,  0.21521) circle [radius=0.0075];
\draw[fill] (  0.12077,  0.23702) circle [radius=0.0075];
\draw[fill] (  0.08220,  0.25300) circle [radius=0.0075];
\draw[fill] (  0.04161,  0.26274) circle [radius=0.0075];
\draw[fill] (  0.00000,  0.26602) circle [radius=0.0075];
\draw[fill] ( -0.04161,  0.26274) circle [radius=0.0075];
\draw[fill] ( -0.08220,  0.25300) circle [radius=0.0075];
\draw[fill] ( -0.12077,  0.23702) circle [radius=0.0075];
\draw[fill] ( -0.15636,  0.21521) circle [radius=0.0075];
\draw[fill] ( -0.18810,  0.18810) circle [radius=0.0075];
\draw[fill] ( -0.21521,  0.15636) circle [radius=0.0075];
\draw[fill] ( -0.23702,  0.12077) circle [radius=0.0075];
\draw[fill] ( -0.25300,  0.08220) circle [radius=0.0075];
\draw[fill] ( -0.26274,  0.04161) circle [radius=0.0075];
\draw[fill] ( -0.26602,  0.00000) circle [radius=0.0075];
\draw[fill] ( -0.26274, -0.04161) circle [radius=0.0075];
\draw[fill] ( -0.25300, -0.08220) circle [radius=0.0075];
\draw[fill] ( -0.23702, -0.12077) circle [radius=0.0075];
\draw[fill] ( -0.21521, -0.15636) circle [radius=0.0075];
\draw[fill] ( -0.18810, -0.18810) circle [radius=0.0075];
\draw[fill] ( -0.15636, -0.21521) circle [radius=0.0075];
\draw[fill] ( -0.12077, -0.23702) circle [radius=0.0075];
\draw[fill] ( -0.08220, -0.25300) circle [radius=0.0075];
\draw[fill] ( -0.04161, -0.26274) circle [radius=0.0075];
\draw[fill] (  0.00000, -0.26602) circle [radius=0.0075];
\draw[fill] (  0.04161, -0.26274) circle [radius=0.0075];
\draw[fill] (  0.08220, -0.25300) circle [radius=0.0075];
\draw[fill] (  0.12077, -0.23702) circle [radius=0.0075];
\draw[fill] (  0.15636, -0.21521) circle [radius=0.0075];
\draw[fill] (  0.18810, -0.18810) circle [radius=0.0075];
\draw[fill] (  0.21521, -0.15636) circle [radius=0.0075];
\draw[fill] (  0.23702, -0.12077) circle [radius=0.0075];
\draw[fill] (  0.25300, -0.08220) circle [radius=0.0075];
\draw[fill] (  0.26274, -0.04161) circle [radius=0.0075];
\draw[fill] (  0.33453,  0.00000) circle [radius=0.0075];
\draw[fill] (  0.33041,  0.05233) circle [radius=0.0075];
\draw[fill] (  0.31816,  0.10338) circle [radius=0.0075];
\draw[fill] (  0.29807,  0.15187) circle [radius=0.0075];
\draw[fill] (  0.27064,  0.19663) circle [radius=0.0075];
\draw[fill] (  0.23655,  0.23655) circle [radius=0.0075];
\draw[fill] (  0.19663,  0.27064) circle [radius=0.0075];
\draw[fill] (  0.15187,  0.29807) circle [radius=0.0075];
\draw[fill] (  0.10338,  0.31816) circle [radius=0.0075];
\draw[fill] (  0.05233,  0.33041) circle [radius=0.0075];
\draw[fill] (  0.00000,  0.33453) circle [radius=0.0075];
\draw[fill] ( -0.05233,  0.33041) circle [radius=0.0075];
\draw[fill] ( -0.10338,  0.31816) circle [radius=0.0075];
\draw[fill] ( -0.15187,  0.29807) circle [radius=0.0075];
\draw[fill] ( -0.19663,  0.27064) circle [radius=0.0075];
\draw[fill] ( -0.23655,  0.23655) circle [radius=0.0075];
\draw[fill] ( -0.27064,  0.19663) circle [radius=0.0075];
\draw[fill] ( -0.29807,  0.15187) circle [radius=0.0075];
\draw[fill] ( -0.31816,  0.10338) circle [radius=0.0075];
\draw[fill] ( -0.33041,  0.05233) circle [radius=0.0075];
\draw[fill] ( -0.33453,  0.00000) circle [radius=0.0075];
\draw[fill] ( -0.33041, -0.05233) circle [radius=0.0075];
\draw[fill] ( -0.31816, -0.10338) circle [radius=0.0075];
\draw[fill] ( -0.29807, -0.15187) circle [radius=0.0075];
\draw[fill] ( -0.27064, -0.19663) circle [radius=0.0075];
\draw[fill] ( -0.23655, -0.23655) circle [radius=0.0075];
\draw[fill] ( -0.19663, -0.27064) circle [radius=0.0075];
\draw[fill] ( -0.15187, -0.29807) circle [radius=0.0075];
\draw[fill] ( -0.10338, -0.31816) circle [radius=0.0075];
\draw[fill] ( -0.05233, -0.33041) circle [radius=0.0075];
\draw[fill] (  0.00000, -0.33453) circle [radius=0.0075];
\draw[fill] (  0.05233, -0.33041) circle [radius=0.0075];
\draw[fill] (  0.10338, -0.31816) circle [radius=0.0075];
\draw[fill] (  0.15187, -0.29807) circle [radius=0.0075];
\draw[fill] (  0.19663, -0.27064) circle [radius=0.0075];
\draw[fill] (  0.23655, -0.23655) circle [radius=0.0075];
\draw[fill] (  0.27064, -0.19663) circle [radius=0.0075];
\draw[fill] (  0.29807, -0.15187) circle [radius=0.0075];
\draw[fill] (  0.31816, -0.10338) circle [radius=0.0075];
\draw[fill] (  0.33041, -0.05233) circle [radius=0.0075];
\draw[fill] (  0.40673,  0.00000) circle [radius=0.0075];
\draw[fill] (  0.40173,  0.06363) circle [radius=0.0075];
\draw[fill] (  0.38683,  0.12569) circle [radius=0.0075];
\draw[fill] (  0.36240,  0.18465) circle [radius=0.0075];
\draw[fill] (  0.32906,  0.23907) circle [radius=0.0075];
\draw[fill] (  0.28760,  0.28760) circle [radius=0.0075];
\draw[fill] (  0.23907,  0.32906) circle [radius=0.0075];
\draw[fill] (  0.18465,  0.36240) circle [radius=0.0075];
\draw[fill] (  0.12569,  0.38683) circle [radius=0.0075];
\draw[fill] (  0.06363,  0.40173) circle [radius=0.0075];
\draw[fill] (  0.00000,  0.40673) circle [radius=0.0075];
\draw[fill] ( -0.06363,  0.40173) circle [radius=0.0075];
\draw[fill] ( -0.12569,  0.38683) circle [radius=0.0075];
\draw[fill] ( -0.18465,  0.36240) circle [radius=0.0075];
\draw[fill] ( -0.23907,  0.32906) circle [radius=0.0075];
\draw[fill] ( -0.28760,  0.28760) circle [radius=0.0075];
\draw[fill] ( -0.32906,  0.23907) circle [radius=0.0075];
\draw[fill] ( -0.36240,  0.18465) circle [radius=0.0075];
\draw[fill] ( -0.38683,  0.12569) circle [radius=0.0075];
\draw[fill] ( -0.40173,  0.06363) circle [radius=0.0075];
\draw[fill] ( -0.40673,  0.00000) circle [radius=0.0075];
\draw[fill] ( -0.40173, -0.06363) circle [radius=0.0075];
\draw[fill] ( -0.38683, -0.12569) circle [radius=0.0075];
\draw[fill] ( -0.36240, -0.18465) circle [radius=0.0075];
\draw[fill] ( -0.32906, -0.23907) circle [radius=0.0075];
\draw[fill] ( -0.28760, -0.28760) circle [radius=0.0075];
\draw[fill] ( -0.23907, -0.32906) circle [radius=0.0075];
\draw[fill] ( -0.18465, -0.36240) circle [radius=0.0075];
\draw[fill] ( -0.12569, -0.38683) circle [radius=0.0075];
\draw[fill] ( -0.06363, -0.40173) circle [radius=0.0075];
\draw[fill] (  0.00000, -0.40673) circle [radius=0.0075];
\draw[fill] (  0.06363, -0.40173) circle [radius=0.0075];
\draw[fill] (  0.12569, -0.38683) circle [radius=0.0075];
\draw[fill] (  0.18465, -0.36240) circle [radius=0.0075];
\draw[fill] (  0.23907, -0.32906) circle [radius=0.0075];
\draw[fill] (  0.28760, -0.28760) circle [radius=0.0075];
\draw[fill] (  0.32906, -0.23907) circle [radius=0.0075];
\draw[fill] (  0.36240, -0.18465) circle [radius=0.0075];
\draw[fill] (  0.38683, -0.12569) circle [radius=0.0075];
\draw[fill] (  0.40173, -0.06363) circle [radius=0.0075];
\draw[fill] (  0.48102,  0.00000) circle [radius=0.0075];
\draw[fill] (  0.47509,  0.07525) circle [radius=0.0075];
\draw[fill] (  0.45747,  0.14864) circle [radius=0.0075];
\draw[fill] (  0.42859,  0.21838) circle [radius=0.0075];
\draw[fill] (  0.38915,  0.28273) circle [radius=0.0075];
\draw[fill] (  0.34013,  0.34013) circle [radius=0.0075];
\draw[fill] (  0.28273,  0.38915) circle [radius=0.0075];
\draw[fill] (  0.21838,  0.42859) circle [radius=0.0075];
\draw[fill] (  0.14864,  0.45747) circle [radius=0.0075];
\draw[fill] (  0.07525,  0.47509) circle [radius=0.0075];
\draw[fill] (  0.00000,  0.48102) circle [radius=0.0075];
\draw[fill] ( -0.07525,  0.47509) circle [radius=0.0075];
\draw[fill] ( -0.14864,  0.45747) circle [radius=0.0075];
\draw[fill] ( -0.21838,  0.42859) circle [radius=0.0075];
\draw[fill] ( -0.28273,  0.38915) circle [radius=0.0075];
\draw[fill] ( -0.34013,  0.34013) circle [radius=0.0075];
\draw[fill] ( -0.38915,  0.28273) circle [radius=0.0075];
\draw[fill] ( -0.42859,  0.21838) circle [radius=0.0075];
\draw[fill] ( -0.45747,  0.14864) circle [radius=0.0075];
\draw[fill] ( -0.47509,  0.07525) circle [radius=0.0075];
\draw[fill] ( -0.48102,  0.00000) circle [radius=0.0075];
\draw[fill] ( -0.47509, -0.07525) circle [radius=0.0075];
\draw[fill] ( -0.45747, -0.14864) circle [radius=0.0075];
\draw[fill] ( -0.42859, -0.21838) circle [radius=0.0075];
\draw[fill] ( -0.38915, -0.28273) circle [radius=0.0075];
\draw[fill] ( -0.34013, -0.34013) circle [radius=0.0075];
\draw[fill] ( -0.28273, -0.38915) circle [radius=0.0075];
\draw[fill] ( -0.21838, -0.42859) circle [radius=0.0075];
\draw[fill] ( -0.14864, -0.45747) circle [radius=0.0075];
\draw[fill] ( -0.07525, -0.47509) circle [radius=0.0075];
\draw[fill] (  0.00000, -0.48102) circle [radius=0.0075];
\draw[fill] (  0.07525, -0.47509) circle [radius=0.0075];
\draw[fill] (  0.14864, -0.45747) circle [radius=0.0075];
\draw[fill] (  0.21838, -0.42859) circle [radius=0.0075];
\draw[fill] (  0.28273, -0.38915) circle [radius=0.0075];
\draw[fill] (  0.34013, -0.34013) circle [radius=0.0075];
\draw[fill] (  0.38915, -0.28273) circle [radius=0.0075];
\draw[fill] (  0.42859, -0.21838) circle [radius=0.0075];
\draw[fill] (  0.45747, -0.14864) circle [radius=0.0075];
\draw[fill] (  0.47509, -0.07525) circle [radius=0.0075];
\draw[fill] (  0.55571,  0.00000) circle [radius=0.0075];
\draw[fill] (  0.54887,  0.08693) circle [radius=0.0075];
\draw[fill] (  0.52852,  0.17173) circle [radius=0.0075];
\draw[fill] (  0.49515,  0.25229) circle [radius=0.0075];
\draw[fill] (  0.44958,  0.32664) circle [radius=0.0075];
\draw[fill] (  0.39295,  0.39295) circle [radius=0.0075];
\draw[fill] (  0.32664,  0.44958) circle [radius=0.0075];
\draw[fill] (  0.25229,  0.49515) circle [radius=0.0075];
\draw[fill] (  0.17173,  0.52852) circle [radius=0.0075];
\draw[fill] (  0.08693,  0.54887) circle [radius=0.0075];
\draw[fill] (  0.00000,  0.55571) circle [radius=0.0075];
\draw[fill] ( -0.08693,  0.54887) circle [radius=0.0075];
\draw[fill] ( -0.17173,  0.52852) circle [radius=0.0075];
\draw[fill] ( -0.25229,  0.49515) circle [radius=0.0075];
\draw[fill] ( -0.32664,  0.44958) circle [radius=0.0075];
\draw[fill] ( -0.39295,  0.39295) circle [radius=0.0075];
\draw[fill] ( -0.44958,  0.32664) circle [radius=0.0075];
\draw[fill] ( -0.49515,  0.25229) circle [radius=0.0075];
\draw[fill] ( -0.52852,  0.17173) circle [radius=0.0075];
\draw[fill] ( -0.54887,  0.08693) circle [radius=0.0075];
\draw[fill] ( -0.55571,  0.00000) circle [radius=0.0075];
\draw[fill] ( -0.54887, -0.08693) circle [radius=0.0075];
\draw[fill] ( -0.52852, -0.17173) circle [radius=0.0075];
\draw[fill] ( -0.49515, -0.25229) circle [radius=0.0075];
\draw[fill] ( -0.44958, -0.32664) circle [radius=0.0075];
\draw[fill] ( -0.39295, -0.39295) circle [radius=0.0075];
\draw[fill] ( -0.32664, -0.44958) circle [radius=0.0075];
\draw[fill] ( -0.25229, -0.49515) circle [radius=0.0075];
\draw[fill] ( -0.17173, -0.52852) circle [radius=0.0075];
\draw[fill] ( -0.08693, -0.54887) circle [radius=0.0075];
\draw[fill] (  0.00000, -0.55571) circle [radius=0.0075];
\draw[fill] (  0.08693, -0.54887) circle [radius=0.0075];
\draw[fill] (  0.17173, -0.52852) circle [radius=0.0075];
\draw[fill] (  0.25229, -0.49515) circle [radius=0.0075];
\draw[fill] (  0.32664, -0.44958) circle [radius=0.0075];
\draw[fill] (  0.39295, -0.39295) circle [radius=0.0075];
\draw[fill] (  0.44958, -0.32664) circle [radius=0.0075];
\draw[fill] (  0.49515, -0.25229) circle [radius=0.0075];
\draw[fill] (  0.52852, -0.17173) circle [radius=0.0075];
\draw[fill] (  0.54887, -0.08693) circle [radius=0.0075];
\draw[fill] (  0.62916,  0.00000) circle [radius=0.0075];
\draw[fill] (  0.62142,  0.09842) circle [radius=0.0075];
\draw[fill] (  0.59837,  0.19442) circle [radius=0.0075];
\draw[fill] (  0.56059,  0.28563) circle [radius=0.0075];
\draw[fill] (  0.50900,  0.36981) circle [radius=0.0075];
\draw[fill] (  0.44489,  0.44489) circle [radius=0.0075];
\draw[fill] (  0.36981,  0.50900) circle [radius=0.0075];
\draw[fill] (  0.28563,  0.56059) circle [radius=0.0075];
\draw[fill] (  0.19442,  0.59837) circle [radius=0.0075];
\draw[fill] (  0.09842,  0.62142) circle [radius=0.0075];
\draw[fill] (  0.00000,  0.62916) circle [radius=0.0075];
\draw[fill] ( -0.09842,  0.62142) circle [radius=0.0075];
\draw[fill] ( -0.19442,  0.59837) circle [radius=0.0075];
\draw[fill] ( -0.28563,  0.56059) circle [radius=0.0075];
\draw[fill] ( -0.36981,  0.50900) circle [radius=0.0075];
\draw[fill] ( -0.44489,  0.44489) circle [radius=0.0075];
\draw[fill] ( -0.50900,  0.36981) circle [radius=0.0075];
\draw[fill] ( -0.56059,  0.28563) circle [radius=0.0075];
\draw[fill] ( -0.59837,  0.19442) circle [radius=0.0075];
\draw[fill] ( -0.62142,  0.09842) circle [radius=0.0075];
\draw[fill] ( -0.62916,  0.00000) circle [radius=0.0075];
\draw[fill] ( -0.62142, -0.09842) circle [radius=0.0075];
\draw[fill] ( -0.59837, -0.19442) circle [radius=0.0075];
\draw[fill] ( -0.56059, -0.28563) circle [radius=0.0075];
\draw[fill] ( -0.50900, -0.36981) circle [radius=0.0075];
\draw[fill] ( -0.44489, -0.44489) circle [radius=0.0075];
\draw[fill] ( -0.36981, -0.50900) circle [radius=0.0075];
\draw[fill] ( -0.28563, -0.56059) circle [radius=0.0075];
\draw[fill] ( -0.19442, -0.59837) circle [radius=0.0075];
\draw[fill] ( -0.09842, -0.62142) circle [radius=0.0075];
\draw[fill] (  0.00000, -0.62916) circle [radius=0.0075];
\draw[fill] (  0.09842, -0.62142) circle [radius=0.0075];
\draw[fill] (  0.19442, -0.59837) circle [radius=0.0075];
\draw[fill] (  0.28563, -0.56059) circle [radius=0.0075];
\draw[fill] (  0.36981, -0.50900) circle [radius=0.0075];
\draw[fill] (  0.44489, -0.44489) circle [radius=0.0075];
\draw[fill] (  0.50900, -0.36981) circle [radius=0.0075];
\draw[fill] (  0.56059, -0.28563) circle [radius=0.0075];
\draw[fill] (  0.59837, -0.19442) circle [radius=0.0075];
\draw[fill] (  0.62142, -0.09842) circle [radius=0.0075];
\draw[fill] (  0.69972,  0.00000) circle [radius=0.0075];
\draw[fill] (  0.69110,  0.10946) circle [radius=0.0075];
\draw[fill] (  0.66547,  0.21623) circle [radius=0.0075];
\draw[fill] (  0.62345,  0.31767) circle [radius=0.0075];
\draw[fill] (  0.56608,  0.41128) circle [radius=0.0075];
\draw[fill] (  0.49478,  0.49478) circle [radius=0.0075];
\draw[fill] (  0.41128,  0.56608) circle [radius=0.0075];
\draw[fill] (  0.31767,  0.62345) circle [radius=0.0075];
\draw[fill] (  0.21623,  0.66547) circle [radius=0.0075];
\draw[fill] (  0.10946,  0.69110) circle [radius=0.0075];
\draw[fill] (  0.00000,  0.69972) circle [radius=0.0075];
\draw[fill] ( -0.10946,  0.69110) circle [radius=0.0075];
\draw[fill] ( -0.21623,  0.66547) circle [radius=0.0075];
\draw[fill] ( -0.31767,  0.62345) circle [radius=0.0075];
\draw[fill] ( -0.41128,  0.56608) circle [radius=0.0075];
\draw[fill] ( -0.49478,  0.49478) circle [radius=0.0075];
\draw[fill] ( -0.56608,  0.41128) circle [radius=0.0075];
\draw[fill] ( -0.62345,  0.31767) circle [radius=0.0075];
\draw[fill] ( -0.66547,  0.21623) circle [radius=0.0075];
\draw[fill] ( -0.69110,  0.10946) circle [radius=0.0075];
\draw[fill] ( -0.69972,  0.00000) circle [radius=0.0075];
\draw[fill] ( -0.69110, -0.10946) circle [radius=0.0075];
\draw[fill] ( -0.66547, -0.21623) circle [radius=0.0075];
\draw[fill] ( -0.62345, -0.31767) circle [radius=0.0075];
\draw[fill] ( -0.56608, -0.41128) circle [radius=0.0075];
\draw[fill] ( -0.49478, -0.49478) circle [radius=0.0075];
\draw[fill] ( -0.41128, -0.56608) circle [radius=0.0075];
\draw[fill] ( -0.31767, -0.62345) circle [radius=0.0075];
\draw[fill] ( -0.21623, -0.66547) circle [radius=0.0075];
\draw[fill] ( -0.10946, -0.69110) circle [radius=0.0075];
\draw[fill] (  0.00000, -0.69972) circle [radius=0.0075];
\draw[fill] (  0.10946, -0.69110) circle [radius=0.0075];
\draw[fill] (  0.21623, -0.66547) circle [radius=0.0075];
\draw[fill] (  0.31767, -0.62345) circle [radius=0.0075];
\draw[fill] (  0.41128, -0.56608) circle [radius=0.0075];
\draw[fill] (  0.49478, -0.49478) circle [radius=0.0075];
\draw[fill] (  0.56608, -0.41128) circle [radius=0.0075];
\draw[fill] (  0.62345, -0.31767) circle [radius=0.0075];
\draw[fill] (  0.66547, -0.21623) circle [radius=0.0075];
\draw[fill] (  0.69110, -0.10946) circle [radius=0.0075];
\draw[fill] (  0.76581,  0.00000) circle [radius=0.0075];
\draw[fill] (  0.75638,  0.11980) circle [radius=0.0075];
\draw[fill] (  0.72833,  0.23665) circle [radius=0.0075];
\draw[fill] (  0.68234,  0.34767) circle [radius=0.0075];
\draw[fill] (  0.61955,  0.45013) circle [radius=0.0075];
\draw[fill] (  0.54151,  0.54151) circle [radius=0.0075];
\draw[fill] (  0.45013,  0.61955) circle [radius=0.0075];
\draw[fill] (  0.34767,  0.68234) circle [radius=0.0075];
\draw[fill] (  0.23665,  0.72833) circle [radius=0.0075];
\draw[fill] (  0.11980,  0.75638) circle [radius=0.0075];
\draw[fill] (  0.00000,  0.76581) circle [radius=0.0075];
\draw[fill] ( -0.11980,  0.75638) circle [radius=0.0075];
\draw[fill] ( -0.23665,  0.72833) circle [radius=0.0075];
\draw[fill] ( -0.34767,  0.68234) circle [radius=0.0075];
\draw[fill] ( -0.45013,  0.61955) circle [radius=0.0075];
\draw[fill] ( -0.54151,  0.54151) circle [radius=0.0075];
\draw[fill] ( -0.61955,  0.45013) circle [radius=0.0075];
\draw[fill] ( -0.68234,  0.34767) circle [radius=0.0075];
\draw[fill] ( -0.72833,  0.23665) circle [radius=0.0075];
\draw[fill] ( -0.75638,  0.11980) circle [radius=0.0075];
\draw[fill] ( -0.76581,  0.00000) circle [radius=0.0075];
\draw[fill] ( -0.75638, -0.11980) circle [radius=0.0075];
\draw[fill] ( -0.72833, -0.23665) circle [radius=0.0075];
\draw[fill] ( -0.68234, -0.34767) circle [radius=0.0075];
\draw[fill] ( -0.61955, -0.45013) circle [radius=0.0075];
\draw[fill] ( -0.54151, -0.54151) circle [radius=0.0075];
\draw[fill] ( -0.45013, -0.61955) circle [radius=0.0075];
\draw[fill] ( -0.34767, -0.68234) circle [radius=0.0075];
\draw[fill] ( -0.23665, -0.72833) circle [radius=0.0075];
\draw[fill] ( -0.11980, -0.75638) circle [radius=0.0075];
\draw[fill] (  0.00000, -0.76581) circle [radius=0.0075];
\draw[fill] (  0.11980, -0.75638) circle [radius=0.0075];
\draw[fill] (  0.23665, -0.72833) circle [radius=0.0075];
\draw[fill] (  0.34767, -0.68234) circle [radius=0.0075];
\draw[fill] (  0.45013, -0.61955) circle [radius=0.0075];
\draw[fill] (  0.54151, -0.54151) circle [radius=0.0075];
\draw[fill] (  0.61955, -0.45013) circle [radius=0.0075];
\draw[fill] (  0.68234, -0.34767) circle [radius=0.0075];
\draw[fill] (  0.72833, -0.23665) circle [radius=0.0075];
\draw[fill] (  0.75638, -0.11980) circle [radius=0.0075];
\draw[fill] (  0.82595,  0.00000) circle [radius=0.0075];
\draw[fill] (  0.81578,  0.12921) circle [radius=0.0075];
\draw[fill] (  0.78553,  0.25523) circle [radius=0.0075];
\draw[fill] (  0.73593,  0.37497) circle [radius=0.0075];
\draw[fill] (  0.66821,  0.48548) circle [radius=0.0075];
\draw[fill] (  0.58404,  0.58404) circle [radius=0.0075];
\draw[fill] (  0.48548,  0.66821) circle [radius=0.0075];
\draw[fill] (  0.37497,  0.73593) circle [radius=0.0075];
\draw[fill] (  0.25523,  0.78553) circle [radius=0.0075];
\draw[fill] (  0.12921,  0.81578) circle [radius=0.0075];
\draw[fill] (  0.00000,  0.82595) circle [radius=0.0075];
\draw[fill] ( -0.12921,  0.81578) circle [radius=0.0075];
\draw[fill] ( -0.25523,  0.78553) circle [radius=0.0075];
\draw[fill] ( -0.37497,  0.73593) circle [radius=0.0075];
\draw[fill] ( -0.48548,  0.66821) circle [radius=0.0075];
\draw[fill] ( -0.58404,  0.58404) circle [radius=0.0075];
\draw[fill] ( -0.66821,  0.48548) circle [radius=0.0075];
\draw[fill] ( -0.73593,  0.37497) circle [radius=0.0075];
\draw[fill] ( -0.78553,  0.25523) circle [radius=0.0075];
\draw[fill] ( -0.81578,  0.12921) circle [radius=0.0075];
\draw[fill] ( -0.82595,  0.00000) circle [radius=0.0075];
\draw[fill] ( -0.81578, -0.12921) circle [radius=0.0075];
\draw[fill] ( -0.78553, -0.25523) circle [radius=0.0075];
\draw[fill] ( -0.73593, -0.37497) circle [radius=0.0075];
\draw[fill] ( -0.66821, -0.48548) circle [radius=0.0075];
\draw[fill] ( -0.58404, -0.58404) circle [radius=0.0075];
\draw[fill] ( -0.48548, -0.66821) circle [radius=0.0075];
\draw[fill] ( -0.37497, -0.73593) circle [radius=0.0075];
\draw[fill] ( -0.25523, -0.78553) circle [radius=0.0075];
\draw[fill] ( -0.12921, -0.81578) circle [radius=0.0075];
\draw[fill] (  0.00000, -0.82595) circle [radius=0.0075];
\draw[fill] (  0.12921, -0.81578) circle [radius=0.0075];
\draw[fill] (  0.25523, -0.78553) circle [radius=0.0075];
\draw[fill] (  0.37497, -0.73593) circle [radius=0.0075];
\draw[fill] (  0.48548, -0.66821) circle [radius=0.0075];
\draw[fill] (  0.58404, -0.58404) circle [radius=0.0075];
\draw[fill] (  0.66821, -0.48548) circle [radius=0.0075];
\draw[fill] (  0.73593, -0.37497) circle [radius=0.0075];
\draw[fill] (  0.78553, -0.25523) circle [radius=0.0075];
\draw[fill] (  0.81578, -0.12921) circle [radius=0.0075];
\draw[fill] (  0.87881,  0.00000) circle [radius=0.0075];
\draw[fill] (  0.86799,  0.13748) circle [radius=0.0075];
\draw[fill] (  0.83580,  0.27157) circle [radius=0.0075];
\draw[fill] (  0.78303,  0.39897) circle [radius=0.0075];
\draw[fill] (  0.71097,  0.51655) circle [radius=0.0075];
\draw[fill] (  0.62141,  0.62141) circle [radius=0.0075];
\draw[fill] (  0.51655,  0.71097) circle [radius=0.0075];
\draw[fill] (  0.39897,  0.78303) circle [radius=0.0075];
\draw[fill] (  0.27157,  0.83580) circle [radius=0.0075];
\draw[fill] (  0.13748,  0.86799) circle [radius=0.0075];
\draw[fill] (  0.00000,  0.87881) circle [radius=0.0075];
\draw[fill] ( -0.13748,  0.86799) circle [radius=0.0075];
\draw[fill] ( -0.27157,  0.83580) circle [radius=0.0075];
\draw[fill] ( -0.39897,  0.78303) circle [radius=0.0075];
\draw[fill] ( -0.51655,  0.71097) circle [radius=0.0075];
\draw[fill] ( -0.62141,  0.62141) circle [radius=0.0075];
\draw[fill] ( -0.71097,  0.51655) circle [radius=0.0075];
\draw[fill] ( -0.78303,  0.39897) circle [radius=0.0075];
\draw[fill] ( -0.83580,  0.27157) circle [radius=0.0075];
\draw[fill] ( -0.86799,  0.13748) circle [radius=0.0075];
\draw[fill] ( -0.87881,  0.00000) circle [radius=0.0075];
\draw[fill] ( -0.86799, -0.13748) circle [radius=0.0075];
\draw[fill] ( -0.83580, -0.27157) circle [radius=0.0075];
\draw[fill] ( -0.78303, -0.39897) circle [radius=0.0075];
\draw[fill] ( -0.71097, -0.51655) circle [radius=0.0075];
\draw[fill] ( -0.62141, -0.62141) circle [radius=0.0075];
\draw[fill] ( -0.51655, -0.71097) circle [radius=0.0075];
\draw[fill] ( -0.39897, -0.78303) circle [radius=0.0075];
\draw[fill] ( -0.27157, -0.83580) circle [radius=0.0075];
\draw[fill] ( -0.13748, -0.86799) circle [radius=0.0075];
\draw[fill] (  0.00000, -0.87881) circle [radius=0.0075];
\draw[fill] (  0.13748, -0.86799) circle [radius=0.0075];
\draw[fill] (  0.27157, -0.83580) circle [radius=0.0075];
\draw[fill] (  0.39897, -0.78303) circle [radius=0.0075];
\draw[fill] (  0.51655, -0.71097) circle [radius=0.0075];
\draw[fill] (  0.62141, -0.62141) circle [radius=0.0075];
\draw[fill] (  0.71097, -0.51655) circle [radius=0.0075];
\draw[fill] (  0.78303, -0.39897) circle [radius=0.0075];
\draw[fill] (  0.83580, -0.27157) circle [radius=0.0075];
\draw[fill] (  0.86799, -0.13748) circle [radius=0.0075];
\draw[fill] (  0.92320,  0.00000) circle [radius=0.0075];
\draw[fill] (  0.91183,  0.14442) circle [radius=0.0075];
\draw[fill] (  0.87801,  0.28528) circle [radius=0.0075];
\draw[fill] (  0.82258,  0.41912) circle [radius=0.0075];
\draw[fill] (  0.74688,  0.54264) circle [radius=0.0075];
\draw[fill] (  0.65280,  0.65280) circle [radius=0.0075];
\draw[fill] (  0.54264,  0.74688) circle [radius=0.0075];
\draw[fill] (  0.41912,  0.82258) circle [radius=0.0075];
\draw[fill] (  0.28528,  0.87801) circle [radius=0.0075];
\draw[fill] (  0.14442,  0.91183) circle [radius=0.0075];
\draw[fill] (  0.00000,  0.92320) circle [radius=0.0075];
\draw[fill] ( -0.14442,  0.91183) circle [radius=0.0075];
\draw[fill] ( -0.28528,  0.87801) circle [radius=0.0075];
\draw[fill] ( -0.41912,  0.82258) circle [radius=0.0075];
\draw[fill] ( -0.54264,  0.74688) circle [radius=0.0075];
\draw[fill] ( -0.65280,  0.65280) circle [radius=0.0075];
\draw[fill] ( -0.74688,  0.54264) circle [radius=0.0075];
\draw[fill] ( -0.82258,  0.41912) circle [radius=0.0075];
\draw[fill] ( -0.87801,  0.28528) circle [radius=0.0075];
\draw[fill] ( -0.91183,  0.14442) circle [radius=0.0075];
\draw[fill] ( -0.92320,  0.00000) circle [radius=0.0075];
\draw[fill] ( -0.91183, -0.14442) circle [radius=0.0075];
\draw[fill] ( -0.87801, -0.28528) circle [radius=0.0075];
\draw[fill] ( -0.82258, -0.41912) circle [radius=0.0075];
\draw[fill] ( -0.74688, -0.54264) circle [radius=0.0075];
\draw[fill] ( -0.65280, -0.65280) circle [radius=0.0075];
\draw[fill] ( -0.54264, -0.74688) circle [radius=0.0075];
\draw[fill] ( -0.41912, -0.82258) circle [radius=0.0075];
\draw[fill] ( -0.28528, -0.87801) circle [radius=0.0075];
\draw[fill] ( -0.14442, -0.91183) circle [radius=0.0075];
\draw[fill] (  0.00000, -0.92320) circle [radius=0.0075];
\draw[fill] (  0.14442, -0.91183) circle [radius=0.0075];
\draw[fill] (  0.28528, -0.87801) circle [radius=0.0075];
\draw[fill] (  0.41912, -0.82258) circle [radius=0.0075];
\draw[fill] (  0.54264, -0.74688) circle [radius=0.0075];
\draw[fill] (  0.65280, -0.65280) circle [radius=0.0075];
\draw[fill] (  0.74688, -0.54264) circle [radius=0.0075];
\draw[fill] (  0.82258, -0.41912) circle [radius=0.0075];
\draw[fill] (  0.87801, -0.28528) circle [radius=0.0075];
\draw[fill] (  0.91183, -0.14442) circle [radius=0.0075];
\draw[fill] (  0.95813,  0.00000) circle [radius=0.0075];
\draw[fill] (  0.94633,  0.14988) circle [radius=0.0075];
\draw[fill] (  0.91123,  0.29608) circle [radius=0.0075];
\draw[fill] (  0.85370,  0.43498) circle [radius=0.0075];
\draw[fill] (  0.77514,  0.56317) circle [radius=0.0075];
\draw[fill] (  0.67750,  0.67750) circle [radius=0.0075];
\draw[fill] (  0.56317,  0.77514) circle [radius=0.0075];
\draw[fill] (  0.43498,  0.85370) circle [radius=0.0075];
\draw[fill] (  0.29608,  0.91123) circle [radius=0.0075];
\draw[fill] (  0.14988,  0.94633) circle [radius=0.0075];
\draw[fill] (  0.00000,  0.95813) circle [radius=0.0075];
\draw[fill] ( -0.14988,  0.94633) circle [radius=0.0075];
\draw[fill] ( -0.29608,  0.91123) circle [radius=0.0075];
\draw[fill] ( -0.43498,  0.85370) circle [radius=0.0075];
\draw[fill] ( -0.56317,  0.77514) circle [radius=0.0075];
\draw[fill] ( -0.67750,  0.67750) circle [radius=0.0075];
\draw[fill] ( -0.77514,  0.56317) circle [radius=0.0075];
\draw[fill] ( -0.85370,  0.43498) circle [radius=0.0075];
\draw[fill] ( -0.91123,  0.29608) circle [radius=0.0075];
\draw[fill] ( -0.94633,  0.14988) circle [radius=0.0075];
\draw[fill] ( -0.95813,  0.00000) circle [radius=0.0075];
\draw[fill] ( -0.94633, -0.14988) circle [radius=0.0075];
\draw[fill] ( -0.91123, -0.29608) circle [radius=0.0075];
\draw[fill] ( -0.85370, -0.43498) circle [radius=0.0075];
\draw[fill] ( -0.77514, -0.56317) circle [radius=0.0075];
\draw[fill] ( -0.67750, -0.67750) circle [radius=0.0075];
\draw[fill] ( -0.56317, -0.77514) circle [radius=0.0075];
\draw[fill] ( -0.43498, -0.85370) circle [radius=0.0075];
\draw[fill] ( -0.29608, -0.91123) circle [radius=0.0075];
\draw[fill] ( -0.14988, -0.94633) circle [radius=0.0075];
\draw[fill] (  0.00000, -0.95813) circle [radius=0.0075];
\draw[fill] (  0.14988, -0.94633) circle [radius=0.0075];
\draw[fill] (  0.29608, -0.91123) circle [radius=0.0075];
\draw[fill] (  0.43498, -0.85370) circle [radius=0.0075];
\draw[fill] (  0.56317, -0.77514) circle [radius=0.0075];
\draw[fill] (  0.67750, -0.67750) circle [radius=0.0075];
\draw[fill] (  0.77514, -0.56317) circle [radius=0.0075];
\draw[fill] (  0.85370, -0.43498) circle [radius=0.0075];
\draw[fill] (  0.91123, -0.29608) circle [radius=0.0075];
\draw[fill] (  0.94633, -0.14988) circle [radius=0.0075];
\draw[fill] (  0.98282,  0.00000) circle [radius=0.0075];
\draw[fill] (  0.97072,  0.15375) circle [radius=0.0075];
\draw[fill] (  0.93472,  0.30371) circle [radius=0.0075];
\draw[fill] (  0.87570,  0.44619) circle [radius=0.0075];
\draw[fill] (  0.79512,  0.57769) circle [radius=0.0075];
\draw[fill] (  0.69496,  0.69496) circle [radius=0.0075];
\draw[fill] (  0.57769,  0.79512) circle [radius=0.0075];
\draw[fill] (  0.44619,  0.87570) circle [radius=0.0075];
\draw[fill] (  0.30371,  0.93472) circle [radius=0.0075];
\draw[fill] (  0.15375,  0.97072) circle [radius=0.0075];
\draw[fill] (  0.00000,  0.98282) circle [radius=0.0075];
\draw[fill] ( -0.15375,  0.97072) circle [radius=0.0075];
\draw[fill] ( -0.30371,  0.93472) circle [radius=0.0075];
\draw[fill] ( -0.44619,  0.87570) circle [radius=0.0075];
\draw[fill] ( -0.57769,  0.79512) circle [radius=0.0075];
\draw[fill] ( -0.69496,  0.69496) circle [radius=0.0075];
\draw[fill] ( -0.79512,  0.57769) circle [radius=0.0075];
\draw[fill] ( -0.87570,  0.44619) circle [radius=0.0075];
\draw[fill] ( -0.93472,  0.30371) circle [radius=0.0075];
\draw[fill] ( -0.97072,  0.15375) circle [radius=0.0075];
\draw[fill] ( -0.98282,  0.00000) circle [radius=0.0075];
\draw[fill] ( -0.97072, -0.15375) circle [radius=0.0075];
\draw[fill] ( -0.93472, -0.30371) circle [radius=0.0075];
\draw[fill] ( -0.87570, -0.44619) circle [radius=0.0075];
\draw[fill] ( -0.79512, -0.57769) circle [radius=0.0075];
\draw[fill] ( -0.69496, -0.69496) circle [radius=0.0075];
\draw[fill] ( -0.57769, -0.79512) circle [radius=0.0075];
\draw[fill] ( -0.44619, -0.87570) circle [radius=0.0075];
\draw[fill] ( -0.30371, -0.93472) circle [radius=0.0075];
\draw[fill] ( -0.15375, -0.97072) circle [radius=0.0075];
\draw[fill] (  0.00000, -0.98282) circle [radius=0.0075];
\draw[fill] (  0.15375, -0.97072) circle [radius=0.0075];
\draw[fill] (  0.30371, -0.93472) circle [radius=0.0075];
\draw[fill] (  0.44619, -0.87570) circle [radius=0.0075];
\draw[fill] (  0.57769, -0.79512) circle [radius=0.0075];
\draw[fill] (  0.69496, -0.69496) circle [radius=0.0075];
\draw[fill] (  0.79512, -0.57769) circle [radius=0.0075];
\draw[fill] (  0.87570, -0.44619) circle [radius=0.0075];
\draw[fill] (  0.93472, -0.30371) circle [radius=0.0075];
\draw[fill] (  0.97072, -0.15375) circle [radius=0.0075];
\draw[fill] (  0.99672,  0.00000) circle [radius=0.0075];
\draw[fill] (  0.98445,  0.15592) circle [radius=0.0075];
\draw[fill] (  0.94794,  0.30800) circle [radius=0.0075];
\draw[fill] (  0.88809,  0.45250) circle [radius=0.0075];
\draw[fill] (  0.80637,  0.58586) circle [radius=0.0075];
\draw[fill] (  0.70479,  0.70479) circle [radius=0.0075];
\draw[fill] (  0.58586,  0.80637) circle [radius=0.0075];
\draw[fill] (  0.45250,  0.88809) circle [radius=0.0075];
\draw[fill] (  0.30800,  0.94794) circle [radius=0.0075];
\draw[fill] (  0.15592,  0.98445) circle [radius=0.0075];
\draw[fill] (  0.00000,  0.99672) circle [radius=0.0075];
\draw[fill] ( -0.15592,  0.98445) circle [radius=0.0075];
\draw[fill] ( -0.30800,  0.94794) circle [radius=0.0075];
\draw[fill] ( -0.45250,  0.88809) circle [radius=0.0075];
\draw[fill] ( -0.58586,  0.80637) circle [radius=0.0075];
\draw[fill] ( -0.70479,  0.70479) circle [radius=0.0075];
\draw[fill] ( -0.80637,  0.58586) circle [radius=0.0075];
\draw[fill] ( -0.88809,  0.45250) circle [radius=0.0075];
\draw[fill] ( -0.94794,  0.30800) circle [radius=0.0075];
\draw[fill] ( -0.98445,  0.15592) circle [radius=0.0075];
\draw[fill] ( -0.99672,  0.00000) circle [radius=0.0075];
\draw[fill] ( -0.98445, -0.15592) circle [radius=0.0075];
\draw[fill] ( -0.94794, -0.30800) circle [radius=0.0075];
\draw[fill] ( -0.88809, -0.45250) circle [radius=0.0075];
\draw[fill] ( -0.80637, -0.58586) circle [radius=0.0075];
\draw[fill] ( -0.70479, -0.70479) circle [radius=0.0075];
\draw[fill] ( -0.58586, -0.80637) circle [radius=0.0075];
\draw[fill] ( -0.45250, -0.88809) circle [radius=0.0075];
\draw[fill] ( -0.30800, -0.94794) circle [radius=0.0075];
\draw[fill] ( -0.15592, -0.98445) circle [radius=0.0075];
\draw[fill] (  0.00000, -0.99672) circle [radius=0.0075];
\draw[fill] (  0.15592, -0.98445) circle [radius=0.0075];
\draw[fill] (  0.30800, -0.94794) circle [radius=0.0075];
\draw[fill] (  0.45250, -0.88809) circle [radius=0.0075];
\draw[fill] (  0.58586, -0.80637) circle [radius=0.0075];
\draw[fill] (  0.70479, -0.70479) circle [radius=0.0075];
\draw[fill] (  0.80637, -0.58586) circle [radius=0.0075];
\draw[fill] (  0.88809, -0.45250) circle [radius=0.0075];
\draw[fill] (  0.94794, -0.30800) circle [radius=0.0075];
\draw[fill] (  0.98445, -0.15592) circle [radius=0.0075];

 \end{tikzpicture}

\caption{
An illustration of locations of Zernike polynomial 
quadrature nodes with $20$ radial nodes
and $40$ angular nodes. }
\label{1f}
\end{figure}
\begin{table}[h!]

  \centering
\scalebox{0.9}{
\begin{tabular}[t]{lc}
   node & $\theta$ \\
  \hline
 $  1$ & $   0.0000000000000000$ \\
 $  2$ & $   0.1570796326794897$ \\
 $  3$ & $   0.3141592653589793$ \\
 $  4$ & $   0.4712388980384690$ \\
 $  5$ & $   0.6283185307179586$ \\
 $  6$ & $   0.7853981633974483$ \\
 $  7$ & $   0.9424777960769379$ \\
 $  8$ & $   1.0995574287564280$ \\
 $  9$ & $   1.2566370614359170$ \\
 $ 10$ & $   1.4137166941154070$ \\
 $ 11$ & $   1.5707963267948970$ \\
 $ 12$ & $   1.7278759594743860$ \\
 $ 13$ & $   1.8849555921538760$ \\
 $ 14$ & $   2.0420352248333660$ \\
 $ 15$ & $   2.1991148575128550$ \\
 $ 16$ & $   2.3561944901923450$ \\
 $ 17$ & $   2.5132741228718340$ \\
 $ 18$ & $   2.6703537555513240$ \\
 $ 19$ & $   2.8274333882308140$ \\
 $ 20$ & $   2.9845130209103030$ \\
 $ 21$ & $   3.1415926535897930$ \\
 $ 22$ & $   3.2986722862692830$ \\
 $ 23$ & $   3.4557519189487720$ \\
 $ 24$ & $   3.6128315516282620$ \\
 $ 25$ & $   3.7699111843077520$ \\
 $ 26$ & $   3.9269908169872410$ \\
 $ 27$ & $   4.0840704496667310$ \\
 $ 28$ & $   4.2411500823462210$ \\
 $ 29$ & $   4.3982297150257100$ \\
 $ 30$ & $   4.5553093477052000$ \\
 $ 31$ & $   4.7123889803846900$ \\
 $ 32$ & $   4.8694686130641790$ \\
 $ 33$ & $   5.0265482457436690$ \\
 $ 34$ & $   5.1836278784231590$ \\
 $ 35$ & $   5.3407075111026480$ \\
 $ 36$ & $   5.4977871437821380$ \\
 $ 37$ & $   5.6548667764616280$ \\
 $ 38$ & $   5.8119464091411170$ \\
 $ 39$ & $   5.9690260418206070$ \\
 $ 40$ & $   6.1261056745000970$ \\
\hline
\end{tabular}
\quad
\begin{tabular}[t]{lc}
   node & $r$ \\
  \hline
 $  1$ & $   0.0083000442070672$ \\
 $  2$ & $   0.0276430533525631$ \\
 $  3$ & $   0.0575344576368137$ \\
 $  4$ & $   0.0973041282065463$ \\
 $  5$ & $   0.1460632469641095$ \\
 $  6$ & $   0.2027224916634053$ \\
 $  7$ & $   0.2660161417643405$ \\
 $  8$ & $   0.3345303010944863$ \\
 $  9$ & $   0.4067344665164935$ \\
 $ 10$ & $   0.4810157112964263$ \\
 $ 11$ & $   0.5557147130369888$ \\
 $ 12$ & $   0.6291628194156031$ \\
 $ 13$ & $   0.6997193231640498$ \\
 $ 14$ & $   0.7658081136864078$ \\
 $ 15$ & $   0.8259528873644578$ \\
 $ 16$ & $   0.8788101326763239$ \\
 $ 17$ & $   0.9231991629103781$ \\
 $ 18$ & $   0.9581285688822349$ \\
 $ 19$ & $   0.9828187818547442$ \\
 $ 20$ & $   0.9967238933309499$ \\

\hline
\end{tabular}

}

\caption{
Locations in the radial and angular directions of Zernike 
polynomial quadrature nodes with $40$ angular nodes
and $20$ radial nodes.
}
\label{7140}
\end{table}
The following remark shows that we can reduce the total 
number of nodes in quadrature rule (\ref{10.7}) while still
integrating the same number of functions.
\begin{remark}
Quadrature rule (\ref{10.7}) integrates all Zernike 
polynomials of order up to $2m-1$ using a tensor product of $2m$ 
equispaced nodes in the angular direction and the $m$ roots of 
$\widetilde{P}_m$ (see \ref{215}) in the radial direction. 
However, for large enough $N$ and small enough $j$, $Z_{N,n}(r_j)$ is
of magnitude smaller than machine precision, where $r_j$ 
denotes the $j^{\text{th}}$ smallest root of $\widetilde{P}_m$. 
As a result, in order to integrate exactly
$Z_{N,n}$ for large $N$, we can use fewer equispaced nodes in 
the angular direction at radius $r_j$. 
\end{remark}
\section{Approximation of Zernike Polynomials}\label{secapprox}
In this section, we describe an interpolation scheme for
Zernike Polynomials.

We will denote by $r_1,...,r_M$ the $M$ roots of $\widetilde{P}_M$ 
(see \ref{215}).

\begin{theorem}\label{thminterp}
Let $M$ be a positive integer and $f:D \rightarrow \R$ be 
a linear combination of Zernike polynomials of degree at 
most $M-1$. That is,
\begin{equation}\label{1100}
f(r,\theta)=\sum_{i,j} \alpha_{i,j}^\ell \overline{Z}_{i,j}^\ell(r,\theta)
\end{equation}
where $i,j$ are non-negative integers satisfying
\begin{equation}\label{1120}
i+2j\leq M-1
\end{equation}
and where $\overline{Z}_{i,j}^\ell(r,\theta)$ is defined by (\ref{10.85})
 and $S_{i}^\ell$ is defined by (\ref{20}). Then,
\begin{equation}\label{1140}
\alpha_{i,j}^\ell=\sum_{k=1}^M \left[ \overline{R}_{i,j}(r_k)\omega_k 
\sum_{l=1}^{2M-1} \frac{2\pi}{2M-1} f(r_k,\theta_l) S_{i}^{\ell} 
(\theta_l)\right]
\end{equation}
where $r_1,...,r_M$ denote the $M$ roots of $\widetilde{P}_M$ 
(see \ref{215}) and $\theta_l$ is defined by the formula
\begin{equation}\label{1160}
\theta_l=l\frac{2\pi}{2M-1}
\end{equation}
for $l=1,2,...,2M-1$.
\end{theorem}
\begin{proof}
Clearly,
\begin{equation}\label{1180}
\alpha_{i,j}^\ell=\int_D f(r,\theta)\overline{Z}_{i,j}^\ell
=\int_{0} ^{2\pi} \int_0^1 f(r,\theta)\overline{R}_{i,j}(r)
S_{i}^\ell(\theta)rdrd\theta.
\end{equation}
Changing the order of integration of (\ref{1180}) and applying
Lemma \ref{60} and Lemma \ref{130}, we obtain
\begin{equation}\label{1200}
\begin{split}
\alpha_{i,j}^\ell&=\int_{0} ^{1} \overline{R}_{i,j}(r) 
r \int_{0}^{2\pi} f(r,\theta)S_{i}^\ell(\theta)d\theta dr \\
&=\int_{0} ^{1} \overline{R}_{i,j}(r) r \sum_{l=1}^{2M-1} 
\frac{2\pi}{2M-1} f(r,\theta_l) S_{i}^\ell(\theta_l)dr.
\end{split}
\end{equation}
Applying Lemma \ref{130} to (\ref{1200}), we obtain
\begin{equation}\label{1210}
\alpha_{i,j}^\ell=\sum_{k=1}^{M} \left[ \overline{R}_{i,j}(r_k) 
\omega_k \sum_{l=1}^{2M-1} 
\frac{2\pi}{2M-1} f(r_k,\theta_l) S_{i,j}^\ell(\theta_l)\right].
\end{equation}
\end{proof}
\begin{remark}
Suppose that $f:D\rightarrow \R$ is a linear combination of 
Zernike polynomials of degree at most $M-1$. 
It follows immediately from Theorem \ref{thminterp} and Theorem 
\ref{lem4.70}
that we can recover exactly the $M^2/2+M/2$ coefficients 
of the Zernike polynomial expanison of $f$ by evaluation of
$f$ at $2M^2-M$ points via (\ref{1140}).
\end{remark}
\begin{remark}
Recovering the $M^2/2+M/2$ coefficients of a Zernike expansion 
of degree at most $M-1$ via (\ref{1210}) requires $O(M^3)$ operations 
by using a FFT to compute the sum
\begin{equation}
\sum_{l=1}^{2M-1} \frac{2\pi}{2M-1} 
f(r,\theta_l) S_{i,j}^\ell(\theta_l)
\end{equation}
and then naively computing the sum
\begin{equation}\label{1240}
\begin{split}
\alpha_{i,j}^\ell=\sum_{k=1}^{M} \overline{R}_{i,j}(r_k) 
\omega_k \sum_{l=1}^{2M-1} 
\frac{2\pi}{2M-1} f(r_k,\theta_l) S_{i,j}^\ell(\theta_l).
\end{split}
\end{equation}
\end{remark}
\begin{remark}
Sum (\ref{1240}) can be computed using an FMM (see, for
example,~\cite{alpert}) which would reduce the evaluation of sum
(\ref{1210}) to a computational cost of $O(M^2 \log(M))$.
\end{remark}
\begin{remark}
Standard interpolation schemes on the unit disk often 
involve representing smooth functions as expansions
in non-smooth functions such as
\begin{equation}\label{1720}
T_n(r)S_N^\ell(\theta)
\end{equation}
where $n$ and $N$ are non-negative integers, $T_n$ is a Chebyshev
polynomial, and $S_N^\ell$ is defined in (\ref{20}).
Such interpolation schemes are amenable to the use of an FFT
in both the angular and radial directions and thus have a 
computational cost of only $O(M^2\log(M))$ for the interpolation
of an $M$-degree Zernike expansion.

However, interpolation scheme (\ref{1140}) has three main advantages
over such a scheme:\\
i) In order to represent a smooth function on the unit disk to full 
precision, a Zernike expansion requires approximately half as many 
terms as an expansion into functions of the form (\ref{1720}) 
(see Figure \ref{8400}).\\
ii) Each function in the interpolated expansion 
is smooth on the disk.\\
iii) The expansion is amenable to filtering.
\end{remark}
\section{Numerical Experiments}\label{secnumres}
The quadrature and interpolation formulas described in Sections
\ref{seczernquad} and \ref{secapprox} were implemented in Fortran 77.  
We used the Lahey/Fujitsu compiler on a 2.9 GHz Intel i7-3520M 
Lenovo laptop. All examples in this section were run in
double precision arithmetic.

In each table in this section, 
the column labeled ``nodes'' denotes the number of nodes in both
the radial and angular direction using quadrature rule (\ref{10.7}).
The column labeled ``exact integral'' denotes the true value of the 
integral being tested. This number is computed using adaptive
gaussian quadrature in extended precision.
The column labeled ``integral via quadrature'' denotes the 
integral approximation using quadrature rule (\ref{10.7}).

We tested the performance of quadrature rule (\ref{10.7}) 
in integrating three different functions over the unit disk. 
In Table \ref{7600} we approximated the integral over the unit
disk of the function $f_1$ defined by the formula
\begin{equation}\label{2100.10}
f_1(x,y)=\frac{1}{1+25(x^2+y^2)}.
\end{equation}
In Table \ref{7700} we use quadrature rule (\ref{10.7}) 
to approximate the integral over the unit
disk of the function $f_2$ defined by the formula
\begin{equation}\label{2200}
f_2(r,\theta)=J_{100}(150r)\cos(100\theta)).
\end{equation}
In Table \ref{7900}, we use quadrature rule (\ref{10.7})
to approximate the integral over the unit disk of the
function $f_3$ defined by the formula
\begin{equation}\label{2400}
f_3(r,\theta)=P_8(x)P_{12}(y).
\end{equation}
We tested the performance of interpolation scheme (\ref{1100})
on two functions defined on the unit disk.

In Figure \ref{8000} we plot the magnitude of 
the coefficients of the Zernike polynomials $R_{0,n}$
for $n=0,1,...,10$ using interpolation scheme (\ref{1100}) 
with $21$ nodes in the radial direction and $41$ in the 
angular direction 
on the function $f_1$ defined in (\ref{2100.10}). 
All coeficients of other terms were of magnitude smaller 
than $10^{-14}$.
In Table \ref{7800} we list the interpolated coefficients of the
Zernike polynomial expansion of the function $f_4$ defined by the 
formula
\begin{equation}
f_4(x,y)=P_2(x)P_4(y)
\end{equation}
where $P_i$ is the $i$th degree Legendre polynomial.
Listed are the coefficients using interpolation 
scheme (\ref{1100}) with $5$ points in the radial direction
and $9$ points in the angular direction
of Zernike polynomials
\begin{equation}
R_{N,n}\cos(N\theta)
\end{equation}
where $N=0,1,...,8$ and $n=0,1,2,3,4$.
All other coefficients were of magnitude smaller than 
$10^{-14}$.
We interpolated the Bessel function 
\begin{equation}\label{9800}
J_{10}(10r)cos(10\theta)
\end{equation}
using interpolation scheme (\ref{1100}) and plot the 
resulting coefficients of the Zernike polynomials 
\begin{equation}\label{9850}
R_{10,n}\cos(10\theta)
\end{equation}
for $n=0,...,16$ in Figure \ref{8400}. All other coefficients
were approximately $0$ to machine precision.
In Figure \ref{8400}, we plot the coefficients of the 
Chebyshev expansion obtained via Chebyshev interpolation 
of the radial component of (\ref{9800}).
\clearpage
\begin{table}[h!]

  \centering
\resizebox{\columnwidth}{!}{%
\begin{tabular}{lllll}
  radial nodes & angular nodes & exact integral & 
  integral via quadrature & relative error \\
  \hline
 $  5 $ & $10$ & $0.4094244859413851$ & $0.4097244673896003$ & $0.732691\times 10^{-3}$ \\
 $ 10$  & $20$ & $0.4094244859413851$ & $0.4094251051077367$ & $0.151228\times 10^{-5}$ \\
 $ 15$  & $30$ & $0.4094244859413851$ & $0.4094244870531256$ & $0.271537\times 10^{-8}$ \\
 $ 20$  & $40$ & $0.4094244859413851$ & $0.4094244859432513$ & $0.455821\times 10^{-11}$ \\
 $ 25$  & $50$ & $0.4094244859413851$ & $0.4094244859413883$ & $0.791759\times 10^{-14}$ \\
 $ 30$  & $60$ & $0.4094244859413851$ & $0.4094244859413848$ & $0.630994\times 10^{-15}$ \\
 $ 35$  & $70$ & $0.4094244859413851$ & $0.4094244859413850$ & $0.142503\times 10^{-15}$ \\
 $ 40$  & $80$ & $0.4094244859413851$ & $0.4094244859413858$ & $0.181146\times 10^{-14}$ \\
  \hline
\end{tabular}
}

  \caption{Quadratures for $f_1(x,y)=(1+25(x^2+y^2))^{-1}$ 
  over the unit disk several different numbers of nodes}
    \label{7600}
\end{table}
\clearpage
\begin{table}[h!]

  \centering
\begin{tabular}{llll}
   radial nodes & angular nodes & exact integral & integral via quadrature  \\
  \hline
 $  5$ & $ 10$ & $0$ & $\hphantom- 0.2670074163846569\times 10^{-1}$ \\
 $ 10$ & $ 20$ & $0$ & $\hphantom- 0.2606355680939063\times 10^{-2}$ \\
 $ 15$ & $ 30$ & $0$ & $\hphantom- 0.3119143925398078\times 10^{-15}$ \\
 $ 20$ & $ 40$ & $0$ & $\hphantom- 0.0000000000000000\times 10^{0}$ \\
 $ 25$ & $ 50$ & $0$ & $\hphantom- 0.3228321977714574\times 10^{-1}$ \\
 $ 30$ & $ 60$ & $0$ & $\hphantom- 0.4945592102178045\times 10^{-16}$ \\
 $ 35$ & $ 70$ & $0$ & $\hphantom- 0.1147861841710902\times 10^{-16}$ \\
 $ 40$ & $ 80$ & $0$ & $\hphantom- 0.8148891073315595\times 10^{-16}$ \\
 $ 45$ & $ 90$ & $0$ & $-0.7432759692263743\times 10^{-16}$ \\
 $ 50$ & $100$ & $0$ & $\hphantom- 0.3207999037057322\times 10^{-1}$ \\
 $ 55$ & $110$ & $0$ & $-0.1399753743762347\times 10^{-15}$ \\
 $ 60$ & $120$ & $0$ & $\hphantom- 0.3075136040459932\times 10^{-16}$ \\
 $ 65$ & $130$ & $0$ & $-0.9458788981593222\times 10^{-16}$ \\
 $ 70$ & $140$ & $0$ & $\hphantom- 0.2045957446273746\times 10^{-17}$ \\
 $ 75$ & $150$ & $0$ & $\hphantom- 0.2416178317504225\times 10^{-16}$ \\
  \hline

\end{tabular}

  \caption{Quadratures for $f_2(r,\theta)=J_{100}(150r)\cos(100\theta)$
   using several different numbers of nodes}
    \label{7700}
\end{table}
\clearpage
\begin{table}[h!]

  \centering
\resizebox{\columnwidth}{!}{%
\begin{tabular}{rrrcl}
   radial nodes & angular nodes & integral via quadrature & exact integral & relative error \\
  \hline
 $  5$&$ 10$ & $-0.8998055487754142\times 10^{-2}$ & $-0.1527947805159123\times 10^{-2}$ & $-0.830191\times 10^{0}$ \\
 $ 10$&$ 20$ & $\hphantom- 0.1655201967553289\times 10^{-1}$ & $-0.1527947805159123\times 10^{-2}$ & $-0.109231\times 10^{1}$ \\
 $ 15$&$ 30$ & $-0.1527947805159138\times 10^{-2}$ & $-0.1527947805159123\times 10^{-2}$ & $-0.979221\times 10^{-14}$ \\
 $ 20$&$ 40$ & $-0.1527947805159132\times 10^{-2}$ & $-0.1527947805159123\times 10^{-2}$ & $-0.567665\times 10^{-14}$ \\
 $ 25$&$ 50$ & $-0.1527947805159108\times 10^{-2}$ & $-0.1527947805159123\times 10^{-2}$ & $\hphantom- 0.102180\times 10^{-13}$ \\
 $ 30$&$ 60$ & $-0.1527947805159144\times 10^{-2}$ & $-0.1527947805159123\times 10^{-2}$ & $-0.134820\times 10^{-13}$ \\
 $ 35$&$ 70$ & $-0.1527947805159128\times 10^{-2}$ & $-0.1527947805159123\times 10^{-2}$ & $-0.269641\times 10^{-14}$ \\
 $ 40$&$ 80$ & $-0.1527947805159155\times 10^{-2}$ & $-0.1527947805159123\times 10^{-2}$ & $-0.210036\times 10^{-13}$ \\
  \hline
\end{tabular}
}
  \caption{Quadratures for $f_3(x,y)=P_8(x)P_{12}(y)$ 
 (see (\ref{2400})) using several different numbers of nodes}
    \label{7900}
\end{table}
\begin{figure}[h!]
\centering
\begin{tikzpicture}

\begin{axis}[
    xlabel={coefficient},
    ylabel={},
    xmin=0, xmax=10,
    ymin=0, ymax=0.3,
    xtick={0,1,2,3,4,5,6,7,8,9,10,11,12},
    ytick={0.1,0.2,0.3,0.4,0.5},
    legend pos=north east,
]
 
\addplot[
    mark=*
    ]
    coordinates 
    {
    (0,0.23)
    (1,0.18)
    (2,0.13)
    (3,0.09)
    (4,0.06)
    (5,0.04)
    (6,0.028)
    (7,0.019)
    (8,0.012)
    (9,0.0087)
    (10,0.006)
    };
 
    \legend{$N=0$}    

\end{axis}

\end{tikzpicture}

\caption{
magnitudes of coefficients of interpolation of 
$f_1(x,y)=(1+25(x^2+y^2))^{-1}$ for $N=0$
}
\label{8000}
\end{figure}
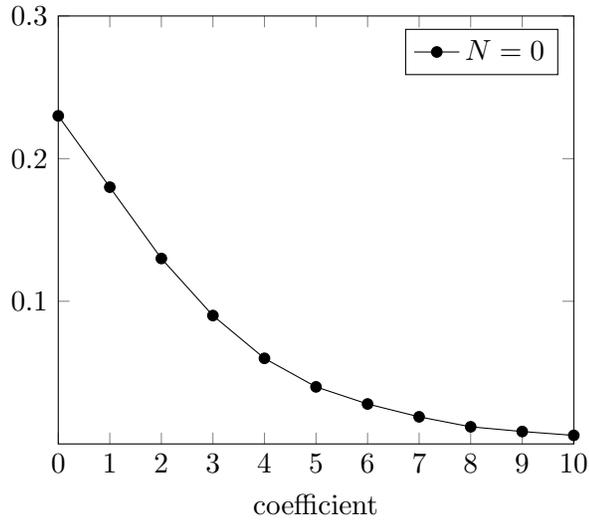
\clearpage
\begin{table}[h!]

  \centering

\resizebox{\columnwidth}{!}{%

\begin{tabular}{llllll}
   $N$ & $n=0$  & $n=1$  & $n=2$  & $n=3$ & $n=4$  \\
  \hline
 $ 0$ & $ \hphantom- 0.02942$ & $\hphantom-0.03297$ & $-0.11998$ & $\hphantom-0.01373$ & 
 $ \hphantom-0.53776 \times 10^{-16}$ \\

 $ 1$ & $ -0.48788 \times 10^{-16}$ & $\hphantom-0.76567 \times 10^{-17} $ & 
 $ \hphantom-0.99670 \times 10^{-18} $ & $ \hphantom-0.22059 \times 10^{-16}$ & - \\

 $ 2$ & $ \hphantom-0.02967$ & $\hphantom-0.11495$ & $-0.00647$ & 
 $-0.90206 \times 10^{-16}$ & - \\

 $ 3$ & $ \hphantom-0.58217 \times 10^{-16}$ & $-0.73297 \times 10^{-16} $ & 
        $\hphantom-0.19321 \times 10^{-17}$ & - & - \\

 $ 4$ & $ \hphantom-0.04926$ & $-0.03238$ & $-0.13010 \times 10^{-16} $ & 
 - & - \\

 $ 5$ & $ \hphantom-0.77604 \times 10^{-16} $ & $\hphantom-0.10474 \times 10^{-15}$ & 
 - & - & - \\

 $ 6$ & $ \hphantom-0.09714$ & $-0.11102 \times 10^{-15}$ & 
 - & - & - \\

 $ 7$ & $ -0.18100 \times 10^{-16} $ & - & - & - & - \\  

 $ 8$ & $ \hphantom-0.77241 \times 10^{-16}  $ & - & - & - & - \\

  \hline

\end{tabular}
}
  \caption{Coefficients of the interpolation of the function 
           $f_4(x,y)=P_2(x)P_4(y)$ into Zernike polynomials of
           degree at most $8$. The entry corresponding to 
           $N,n$ is the coefficient of $R_{N,n}\cos(N\theta)$.
           }
    \label{7800}
\end{table}
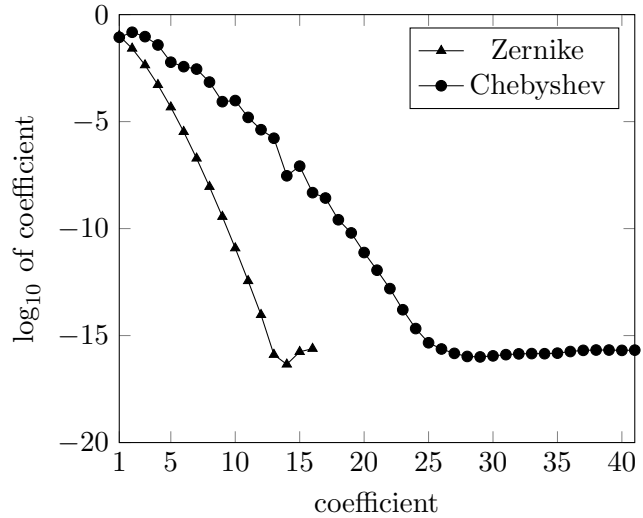
\begin{figure}[h!]
\centering
\begin{tikzpicture}

\begin{axis}[
    xlabel={coefficient},
    ylabel={$\log_{10}$ of coefficient},
    xmin=1, xmax=41,
    ymin=-20, ymax=0,
    xtick={1,5,10,15,20,25,30,35,40,45},
    ytick={0,-5,-10,-15,-20},
    legend pos=north east,
]
 
\addplot[
    mark=triangle*
    ]
    coordinates 
    {
     (  1,      -0.9898973637462893)
     (  2,      -1.5819586550661100)
     (  3,      -2.3588831162234810)
     (  4,      -3.2818384538878600)
     (  5,      -4.3265899327402020)
     (  6,      -5.4763255881423960)
     (  7,      -6.7185889712528420)
     (  8,      -8.0437256952401060)
     (  9,      -9.4440037263722940)
     ( 10,     -10.9130714107455100)
     ( 11,     -12.4457655714899900)
     ( 12,     -14.0303912755872500)
     ( 13,     -15.8930077195246900)
     ( 14,     -16.3457963962040700)
     ( 15,     -15.7586036824183800)
     ( 16,     -15.6123213406515100)
    };
 
 \addplot[
    mark=*
    ]
    coordinates 
    {
     (  1,      -1.0595713478383200)
     (  2,      -0.8238192166464383)
     (  3,      -1.0293346132799840)
     (  4,      -1.4184210516931150)
     (  5,      -2.2236928015527780)
     (  6,      -2.4320890160470980)
     (  7,      -2.5458791787004040)
     (  8,      -3.1529353540202170)
     (  9,      -4.0684262705269870)
     ( 10,      -4.0164562320349430)
     ( 11,      -4.8033818867489570)
     ( 12,      -5.3788871564204650)
     ( 13,      -5.7771701061737610)
     ( 14,      -7.5330217317784680)
     ( 15,      -7.0809635324045650)
     ( 16,      -8.3203906346855500)
     ( 17,      -8.5720716277244970)
     ( 18,      -9.5869646653661940)
     ( 19,     -10.2017710131164700)
     ( 20,     -11.1192420938158200)
     ( 21,     -11.9441235125173900)
     ( 22,     -12.8058618265459900)
     ( 23,     -13.7926138070249200)
     ( 24,     -14.6726336554743400)
     ( 25,     -15.3358881075438200)
     ( 26,     -15.6272308358046700)
     ( 27,     -15.8313508184606000)
     ( 28,     -15.9754399091641200)
     ( 29,     -15.9973417506119500)
     ( 30,     -15.9478563875320300)
     ( 31,     -15.8915380244439100)
     ( 32,     -15.8549738638070200)
     ( 33,     -15.8416916517988200)
     ( 34,     -15.8469558917911700)
     ( 35,     -15.8212504915562700)
     ( 36,     -15.7458293943819500)
     ( 37,     -15.6907318775671300)
     ( 38,     -15.6753007742882200)
     ( 39,     -15.6761934662405600)
     ( 40,     -15.6916568927877700)
     ( 41,     -15.6802334440423000)
    };

    \legend{Zernike, Chebyshev}    

\end{axis}

\end{tikzpicture}

\caption{
Coefficients of the Zernike expansion for $N=10$ of 
$J_{10}(10r)\cos(10\theta)$ using Chebyshev and 
Zernike interpolation in the radial direction with 
$81$ points in the angular direction and $41$ points
in the radial direction.
}
\label{8400}
\end{figure}
\clearpage
\section{Appendix A: Mathematical Properties of Zernike Polynomials}
In this appendix, we define Zernike polynomials in $\R^{p+2}$
and describe some of their basic properties. Zernike polynomials,
denoted $Z_{N,n}^\ell$, are a sequence of orthogonal polynomials 
defined via the formula
  \begin{align}
Z_{N,n}^\ell(x) = R_{N,n}(\|x\|)
  S_N^\ell(x/\|x\|),
    \label{6.10}
  \end{align}
for all $x\in \R^{p+2}$ such that $\|x\| \le 1$, 
where $N$ and $n$ are nonnegative
integers, $S_N^\ell$ are the orthonormal surface harmonics of degree $N$ 
(see Appendix C), and $R_{N,n}$ are polynomials 
of degree $2n+N$ defined via the
formula
  \begin{align}
R_{N,n}(x) = x^N \sum_{m=0}^n (-1)^m {n+N+\frac{p}{2} \choose m} 
  {n\choose m} (x^2)^{n-m} (1-x^2)^m,
    \label{6.20}
  \end{align}
for all $0\le x\le 1$. The polynomials $R_{N,n}$ satisfy the relation
  \begin{align}
R_{N,n}(1) = 1,
    \label{6.30}
  \end{align}
and are orthogonal with respect to the weight function $w(x) = x^{p+1}$, so
that
  \begin{align}
\int_0^1 R_{N,n}(x)R_{N,m}(x) x^{p+1}\, dx = \frac{\delta_{n,m}}
  {2(2n+N+\frac{p}{2}+1)},
    \label{6.40}
  \end{align}
where
  \begin{align}
\delta_{n,m} = \left\{
  \begin{array}{ll}
  1 & \mbox{if $n = m$}, \\
  0 & \mbox{if $n \ne m$}.
  \end{array}
\right.
    \label{6.50}
  \end{align}
We define the polynomials $\overline{R}_{N,n}$ via the formula
  \begin{align}
\overline{R}_{N,n}(x) = \sqrt{2(2n+N+p/2+1)} R_{N,n}(x),
    \label{6.60}
  \end{align}
so that
  \begin{align}
\int_0^1 \bigl(\overline{R}_{N,n}(x)\bigr)^2 x^{p+1}\, dx = 1,
    \label{6.70}
  \end{align}
where $N$ and $n$ are nonnegative integers.
We define the normalized Zernike polynomial, 
$\overline{Z}_{N,n}^\ell$, by the formula
\begin{equation}\label{10.85.2}
\overline{Z}_{N,n}(x)=\overline{R}_{N,n}(\|x\|)S_{N}^\ell(x/\|x\|)
\end{equation}
for all $x\in \R^{p+2}$ such that $\|x\| \le 1$, 
where $N$ and $n$ are nonnegative
integers, $S_N^\ell$ are the orthonormal surface harmonics of degree $N$ 
(see Appendix C), and $R_{N,n}$ is defined in (\ref{6.20}).
We observe that $\overline{Z}_{N,n}^\ell$
has $L^2$ norm of $1$ on the unit ball in $\R^{p+2}$.

In an abuse of notation, we refer to both 
the polynomials $Z^\ell_{N,n}$ and the radial polynomials
$R_{N,n}$  as Zernike polynomials where the meaning is obvious.
\begin{remark}
When $p=-1$, the Zernike polynomials take the form
  \begin{align}
&Z_{0,n}^1(x) = R_{0,n}(|x|) = P_{2n}(x), \label{6.90} \\
&Z_{1,n}^2(x) = \sgn(x)\cdot R_{1,n}(|x|) = P_{2n+1}(x), \label{6.100}
  \end{align}
for $-1\le x\le 1$ and nonnegative integer $n$, where $P_n$ denotes the
Legendre polynomial of degree $n$ and
  \begin{align}
  \sgn(x) = \left\{
    \begin{array}{ll}
    1  & \mbox{if $x > 0$}, \\
    0  & \mbox{if $x = 0$}, \\
    -1 & \mbox{if $x < 0$},
    \end{array}
    \right.
    \label{6.110}
  \end{align}
for all real $x$.
\end{remark}

\begin{remark} \label{rem6.1}
When $p=0$, the Zernike polynomials take the form
  \begin{align}
&Z_{N,n}^1(x_1,x_2) 
  = R_{N,n}(r) \cos(N\theta)/\sqrt{\pi}, \label{6.71}  \\
&Z_{N,n}^2(x_1,x_2)  
  = R_{N,n}(r) \sin(N\theta)/\sqrt{\pi}, \label{6.80}
  \end{align}
where $x_1=r\cos(\theta)$, $x_2=r\sin(\theta)$, and $N$ and $n$ are
nonnegative integers. 
\end{remark}

\subsection{Special Values} \label{sect6.2}

The following formulas are valid for all nonnegative integers $N$ and $n$, and
for all $0\le x\le 1$.
  \begin{align}
& R_{N,0}(x) = x^N, \label{6.180} \\
& R_{N,1}(x) = x^N \bigl((N+p/2+2)x^2 - (N+p/2+1)\bigr), \label{6.190} \\
& R_{N,n}(1) = 1, \label{6.200} \\
& R_{N,n}^{(k)}(0) =0 \quad \mbox{for $k=0,1,\ldots,N-1$}, \label{6.210} \\
& R_{N,n}^{(N)}(0) = (-1)^n N! {n+N+\frac{p}{2} \choose n}. \label{6.220}
  \end{align}

\subsection{Hypergeometric Function} \label{sect6.3}

The polynomials $R_{N,n}$ are related to the hypergeometric function $_2 F_1$
(see~\cite{abramowitz}) by the formula
  \begin{align}
&\hspace*{-3em} R_{N,n}(x) = (-1)^n {n+N+\frac{p}{2} \choose n} x^N
  {_2 F_1}\Bigl(-n, n+N+\frac{p}{2}+1; N+\frac{p}{2}+1; x^2\Bigr),
    \label{6.230}
  \end{align}
where $0\le x\le 1$, and $N$ and $n$ are nonnegative integers.

\subsection{Interrelations} \label{sect6.4}

The polynomials $R_{N,n}$ are related to the Jacobi polynomials via the
formula
  \begin{equation}
R_{N,n}(x) = (-1)^n x^N P_n^{(N+\frac{p}{2},0)}(1-2x^2),
    \label{6.240.2}
  \end{equation}
where $0\le x\le 1$, $N$ and $n$ are nonnegative integers, and
$P^{(\alpha,\beta)}_n$, $\alpha, \beta > -1$, denotes the Jacobi polynomials
of degree $n$ (see~\cite{abramowitz}).

When $p=-1$, the polynomials $R_{N,n}$ are related to the Legendre
polynomials via the formulas
  \begin{align}
R_{0,n}(x) &= P_{2n}(x), \label{6.250} \\
R_{1,n}(x) &= P_{2n+1}(x), \label{6.260}
  \end{align}
where $0\le x\le 1$, $n$ is a nonnegative integer, and $P_n$ denotes the
Legendre polynomial of degree $n$ (see~\cite{abramowitz}).

\subsection{Limit Relations} \label{sect6.5}

The asymptotic behavior of the Zernike polynomials near $0$ as the index $n$
tends to infinity is described by the formula
  \begin{equation}
\lim_{n\to\infty} \frac{(-1)^n R_{N,n}\bigl(\frac{x}{2n}\bigr)}
{(2n)^{p/2}}
= \frac{J_{N+p/2}(x)}{x^{p/2}},
    \label{6.270}
  \end{equation}
where $0\le x \le 1$, $N$ is a nonnegative integer, and $J_\nu$
denotes the Bessel functions of the first kind
(see~\cite{abramowitz}).

\subsection{Zeros} \label{sect6.6}

The asymptotic behavior of the zeros of the polynomials $R_{N,n}$ as $n$
tends to infinity is described by the following relation.  Let
$x_{N,m}^{(n)}$ be the $m$th positive zero of $R_{N,n}$, so that $0 <
x_{N,1}^{(n)} < x_{N,2}^{(n)} < \ldots$. Likewise, let $j_{\nu,m}$ be the
$m$th positive zero of $J_\nu$, so that $0 < j_{\nu,1} < j_{\nu,2} < \ldots$,
where $J_\nu$ denotes the Bessel
functions of the first kind (see~\cite{abramowitz}). Then
  \begin{equation}
\lim_{n\to\infty} 2n x_{N,m}^{(n)} = j_{N+p/2,m},
    \label{6.280}
  \end{equation}
for any nonnegative integer $N$.

\subsection{Inequalities} \label{sect6.7}

The inequality
  \begin{equation}
|R_{N,n}(x)| \le {n+N+\frac{p}{2} \choose n}
    \label{6.290}
  \end{equation}
holds for $0\le x\le 1$ and nonnegative integer $N$ and $n$.

\subsection{Integrals} \label{sect6.8}

The polynomials $R_{N,n}$ satify the relation
  \begin{equation}
\int_0^1 \frac{J_{N+p/2}(xy)}{(xy)^{p/2}} R_{N,n}(y) y^{p+1}\, dy
= \frac{(-1)^n J_{N+p/2+2n+1}(x)}{x^{p/2+1}},
    \label{6.300}
  \end{equation}
where $x\ge 0$, $N$ and $n$ are nonnegative integers, and $J_\nu$
denotes the Bessel functions of the
first kind.

\subsection{Generating Function} \label{sect6.9}

The generating function associated with the polynomials $R_{N,n}$ is given by
the formula
  \begin{equation}
\frac{\bigl(1+z - \sqrt{1+2z(1-2x^2)+z^2}\bigr)^{N+p/2}}
  {(2zx)^{N+p/2} x^{p/2} \sqrt{1+2z(1-2x^2)+z^2}}
= \sum_{n=0}^\infty R_{N,n}(x) z^n,
    \label{6.310}
  \end{equation}
where $0\le x\le 1$ is real, $z$ is a complex number such that
$|z|\le 1$, and $N$ is a nonnegative integer.

\subsection{Differential Equation} \label{sect6.10}

The polynomials $R_{N,n}$ satisfy the differential equation
  \begin{equation}
(1-x^2)y''(x) - 2xy'(x) + \biggl(\chi_{N,n} 
+ \frac{\frac{1}{4} - (N+\frac{p}{2})^2}{x^2}\biggr)y(x) = 0,
    \label{6.320}
  \end{equation}
where
  \begin{align}
\chi_{N,n} = (N+\tfrac{p}{2}+2n+\tfrac{1}{2})
(N+\tfrac{p}{2}+2n+\tfrac{3}{2}),
    \label{6.330}
  \end{align}
and
  \begin{align}
y(x) = x^{p/2+1} R_{N,n}(x),
    \label{6.340}
  \end{align}
for all $0 < x < 1$ and nonnegative integers $N$ and $n$.

\subsection{Recurrence Relations} \label{sect6.11}

The polynomials $R_{N,n}$ satisfy the recurrence relation
  \begin{align}
&\hspace*{-4em} 2(n+1)(n+N+\tfrac{p}{2}+1)(2n+N+\tfrac{p}{2})R_{N,n+1}(x) 
  \notag \\
&= -\bigl( (2n+N+\tfrac{p}{2}+1)(N+\tfrac{p}{2})^2 + (2n+N+\tfrac{p}{2})_3
  (1-2x^2)\bigr) R_{N,n}(x) \notag \\
&\hspace*{2em} - 2n(n+N+\tfrac{p}{2})(2n+N+\tfrac{p}{2}+2) R_{N,n-1}(x),
    \label{6.350.2}
  \end{align}
where $0\le x\le 1$, $N$ is a nonnegative integer, $n$ is a positive integer,
and $(\cdot)_n$ is defined via the formula
  \begin{align}
  (x)_n = x(x+1)(x+2)\ldots (x+n-1),
    \label{6.360}
  \end{align}
for real $x$ and nonnegative integer $n$.  The polynomials $R_{N,n}$ also
satisfy the recurrence relations
  \begin{align}
&\hspace*{-4em} (2n+N+\tfrac{p}{2}+2)x R_{N+1,n}(x) 
= (n+N+\tfrac{p}{2}+1) R_{N,n}(x)
+ (n+1) R_{N,n+1}(x),
    \label{6.370}
  \end{align}
for nonnegative integers $N$ and $n$, and
  \begin{align}
(2n+N+\tfrac{p}{2})x R_{N-1,n}(x) = (n+N+\tfrac{p}{2}) R_{N,n}(x) 
+ n R_{N,n-1}(x),
    \label{6.380}
  \end{align}
for integers $N\ge 1$ and $n\ge 0$, where $0\le x\le 1$.

\subsection{Differential Relations} \label{sect6.12}

The Zernike polynomials satisfy the differential relation given by
the formula
  \begin{align}
&\hspace*{-4em} (2n+N+\tfrac{p}{2})x(1-x^2)\frac{d}{dx} R_{N,n}(x) 
  \notag \\
&= \bigl( N(2n+N+\tfrac{p}{2})+2n^2 - (2n+N)(2n+N+\tfrac{p}{2})x^2 \bigr)
  R_{N,n}(x) \notag \\
&\hspace*{4em} + 2n(n+N+\tfrac{p}{2})R_{N,n-1}(x),
    \label{6.390}
  \end{align}
where $0< x< 1$, $N$ is a nonnegative integer, and $n$ is a positive
integer.

\section{Appendix B: Numerical Evaluation of Zernike Polynomials in 
$\R^{p+2}$}
The main analytical tool of this section is Lemma \ref{lem750} 
which provides a recurrence relation that can be used for 
the evaluation of radial Zernike Polynomials, $R_{N,n}$.

According to \cite{abramowitz}, radial Zernike polynomials, 
$R_{N,n}$, are related to Jacobi polynomials via the formula
\begin{equation}    \label{6.240}
R_{N,n}(x) = (-1)^n x^N P_n^{(N+\frac{p}{2},0)}(1-2x^2),
\end{equation}
where $0\le x\le 1$, $N$ and $n$ are nonnegative integers, and
$P^{(\alpha,0)}_n$ is defined in (\ref{192}).

The following lemma provides a relation that can be used to 
evaluate the polynomial $R_{N,n}$.
\begin{lemma}\label{lem750}
The polynomials $R_{N,n}$ satisfy the recurrence relation
  \begin{align}
&\hspace*{-4em} 2(n+1)(n+N+\tfrac{p}{2}+1)(2n+N+\tfrac{p}{2})R_{N,n+1}(x) 
  \notag \\
&= -\bigl( (2n+N+\tfrac{p}{2}+1)(N+\tfrac{p}{2})^2 + (2n+N+\tfrac{p}{2})_3
  (1-2x^2)\bigr) R_{N,n}(x) \notag \\
&\hspace*{2em} - 2n(n+N+\tfrac{p}{2})(2n+N+\tfrac{p}{2}+2) R_{N,n-1}(x),
    \label{6.350}
  \end{align}
where $0\le x\le 1$, $N$ is a nonnegative integer, $n$ is a 
positive integer, and $(\cdot)_n$ is defined via the formula
  \begin{align}
  (x)_n = x(x+1)(x+2)\ldots (x+n-1),
    \label{6.360.2}
  \end{align}
for real $x$ and nonnegative integer $n$.
\end{lemma}
\begin{proof}
It is well known that the Jacobi polynomial 
$P_n^{(\alpha,0)}(x)$ satisfies the recurrence relation
\begin{equation}\label{230.2}
a_{1n}P_{n+1}^{(\alpha,0)}=(a_{2n}+a_{3n}x)P_n^{(\alpha,0)}(x)
-a_{4n}P_{n-1}^{(\alpha,0)}(x)
\end{equation}
where 
\begin{equation}\label{240.2}
\begin{aligned}
a_{1n}&=2(n+1)(n+\alpha+1)(2n+\alpha)\\
a_{2n}&=(2n+\alpha+1)\alpha^2\\
a_{3n}&=(2n+\alpha)(2n+\alpha+1)(2n+\alpha+2)\\
a_{4n}&=2(n+\alpha)(n)(2n+\alpha+2)
\end{aligned}
\end{equation}
Identity (\ref{6.350}) follows immediately from the combination
of (\ref{230.2}) and (\ref{240.2}).
\end{proof}
\section{Appendix C: Spherical Harmonics in $\R^{p+2}$}
Suppose that $S^{p+1}$ denotes the unit sphere in $\R^{p+2}$. The
spherical harmonics are a set of real-valued continuous functions on
$S^{p+1}$, which are orthonormal and complete in $L^2(S^{p+1})$. The
spherical harmonics of degree $N\ge 0$ are denoted by $S_N^1, S_N^2,
\ldots, \allowbreak S_N^\ell, \ldots, S_N^{h(N)}\colon S^{p+1} \to \R$, 
where
  \begin{align}
h(N)=(2N+p) \frac{(N+p-1)!} {p!\,N!},
  \end{align}
for all nonnegative integers $N$.

The following theorem defines the spherical harmonics as the 
values of certain harmonic, homogeneous polynomials on the sphere 
(see, for example,~\cite{batemanII}).

\begin{theorem}
For each spherical harmonic $S_N^\ell$, where $N\ge 0$ and $1\le \ell
\le h(N)$ are integers, there
exists a polynomial $K_N^\ell \colon \R^{p+2} \to \R$ which is 
harmonic, i.e.
  \begin{align}
\nabla^2 K_N^\ell(x) = 0,
  \end{align}
for all $x\in \R^{p+2}$, and homogenous of degree $N$, i.e.
  \begin{align}
K_N^\ell(\lambda x) = \lambda^N K_N^\ell(x),
  \end{align}
for all $x \in \R^{p+2}$ and $\lambda\in \R$, such that
  \begin{align}
S_N^\ell(\xi) = K_N^\ell(\xi),
  \end{align}
for all $\xi \in S^{p+1}$.

\end{theorem}

The following theorem is proved in, for example,~\cite{batemanII}.

\begin{theorem}
Suppose that $N$ is a nonnegative integer. Then
there are exactly
  \begin{align}
(2N+p) \frac{(N+p-1)!} {p!\,N!}
  \end{align}
linearly independent, harmonic, homogenous polynomials of 
degree $N$ in
$\R^{p+2}$.

\end{theorem}

The following theorem states that for any orthogonal matrix
$U$, the function $S_N^\ell(U\xi)$ is expressible as a linear combination
of $S_N^1(\xi), S_N^2(\xi), \ldots, S_N^{h(N)}(\xi)$ (see,
for example,~\cite{batemanII}).
\begin{theorem}
  \label{spher.rot}
Suppose that $N$ is a nonnegative integer, and that
$S_N^1,S_N^2,\ldots,S_N^{h(N)} \colon S^{p+1} \to \R$ are a
complete set of orthonormal spherical harmonics of degree $N$.
Suppose further that $U$ is a real orthogonal matrix
of dimension $p+2 \times p+2$. Then, for each integer
$1\le \ell \le h(N)$,
there exists real numbers $v_{\ell,1},v_{\ell,2},\ldots,v_{\ell,h(N)}$
such that
  \begin{align}
S_N^\ell(U\xi) = \sum_{k=1}^{h(N)} v_{\ell,k} S_N^k(\xi),
  \end{align}
for all $\xi \in S^{p+1}$. Furthermore, if $V$ is the $h(N)
\times h(N)$ real matrix with elements $v_{i,j}$ for all $1\le i,j \le
h(N)$, then $V$ is also orthogonal.

\end{theorem}

\begin{remark}
From Theorem~(\ref{spher.rot}), we observe that the space of linear combinations
of functions $S_N^\ell$ is invariant under all rotations and reflections
of $S^{p+1}$.

\end{remark}

The following theorem states that if an integral operator acting on the
space of functions $S^{p+1}\to \R$ has a kernel depending only on the inner
product, then the spherical harmonics are eigenfunctions of that
operator (see, for example,~\cite{batemanII}).
\begin{theorem}[Funk-Hecke]
Suppose that $F\colon [-1,1] \to \R$ is a continuous function, and that
$S_N\colon S^{p+1}\to \R$ is any spherical harmonic of degree $N$.
Then
  \begin{align}
\int_\Omega F(\inner{\xi}{\eta}) S_N(\xi) \, d\Omega(\xi) = 
  \lambda_N S_N(\eta),
  \end{align}
for all $\eta\in S^{p+1}$, where $\inner{\cdot}{\cdot}$ denotes
the inner product in $\R^{p+2}$,  the integral is taken
over the whole area of the hypersphere $\Omega$, and
$\lambda_N$ depends only on the function $F$. 
\end{theorem}
\section{Appendix D: The Shifted Jacobi Polynomials $P_n^{(k,0)}(2x-1)$}
\label{secxkpol}
In this section, we introduce a class of Jacobi polynomials
that can be used as quadrature and interpolation nodes for 
Zernike polynomials in $\R^{p+2}$. 

We define $\widetilde{P}_n^k(x)$ to be the shifted Jacobi 
polynomials on the interval $[0,1]$ defined by the formula
\begin{equation}\label{760.2}
\widetilde{P}_n^k(x)=\sqrt{k+2n+1}P_n^{(k,0)}(1-2x)
\end{equation}
where $k>-1$ is a real number and where $P_n^{(k,0)}$ is 
defined in (\ref{192}). It follows immediately from (\ref{760.2}) 
that $\widetilde{P}_n^k(x)$ are orthogonal with respect to 
weight function $x^k$. That is, for all non-negative 
integers $n$, the Jacobi polynomial $\widetilde{P}_n^k$ is a 
polynomial of degree $n$ such that
\begin{equation}\label{740.2}
\int_0^1\widetilde{P}_i^k(x)\widetilde{P}_j^k(x)x^kdx=\delta_{i,j}
\end{equation}
for all non-negative integers $i,j$ where $k>-1$.\\
The following lemma, which follows immediately from the 
combination of Lemma \ref{205} and (\ref{760.2}), 
provides a differential equation satisfied by $\widetilde{P}_n^k$.
\begin{lemma}\label{780}
Let $k>-1$ be a real number and let $n$ be a non-negative
integer. Then, $\widetilde{P}_n^k$ satisfies the differential equation,
\begin{equation}\label{800}
r-r^2\widetilde{P}_n^{k\prime\prime}(r)
+(k-rk+1-2r)\widetilde{P}_n^{k\prime}(r)+
n(n+k+1)\widetilde{P}_n^k(r)=0.
\end{equation}
for all $r\in (0,1)$.
\end{lemma}
The following recurrence for $\widetilde{P}_n^k$ follows 
readily from the combination of Lemma \ref{760.2} and 
(\ref{195}).
\begin{lemma}\label{900.2}
For all non-negative integers $n$ and for all real numbers $k>-1$,
\begin{equation}\label{920}
\begin{split}
\hspace*{-6em}
\widetilde{P}_{n+1}^{k}(r)&=
\frac{(2n+N+1)N^2+(2n+N)(2n+N+1)(2n+N+2)(1-2r)}{2(n+1)(n+N+1)(2n+N)}\\
&\cdot \frac{\sqrt{2n+k+1}}{\sqrt{2(n+1)+k+1}}
\widetilde{P}_n^{k}(r)\\
&-\frac{2(n+N)(n)(2n+N+2)}{2(n+1)(n+N+1)(2n+N)}
\frac{\sqrt{2(n-1)+k+1}}{\sqrt{2(n+1)+k+1}}
\widetilde{P}_{n-1}^{k}(r)
\end{split}
\end{equation}
\end{lemma}
\subsection{Numerical Evaluation of the Shifted Jacobi Polynomials}\label{secxkeval}
The following observations provide a way to evaluate 
$\widetilde{P}_n^k$ and its derivatives.
\begin{observation}\label{765}
Combining (\ref{195}) with (\ref{760.2}), we observe that 
$\widetilde{P}_n^k(x)$ can be evaluated by 
first evaluating $P_n^{(k,0)}(1-2x)$ via recurrence relation 
(\ref{195}) and then multiplying the resulting 
number by
\begin{equation}
\sqrt{k+2n+1}.
\end{equation}
\end{observation}
\begin{observation}\label{767}
Combining (\ref{197}) with (\ref{760.2}), we observe that 
the polynomial $\widetilde{P}_n^{k\prime}(x)$ (see (\ref{760.2})), 
can be evaluated by first evaluating $P_n^{(k,0)\prime}(1-2x)$ 
via recurrence relation (\ref{197}) and then multiplying
the resulting number by 
\begin{equation}
-2\sqrt{k+2n+1}.
\end{equation}
\end{observation}
\subsection{Pr{\"u}fer Transform}\label{secprufer}
In this section, we describe the Pr{\"u}fer Transform, 
which will be used in Section \ref{secxkquad}. A more detailed 
description of the Pr{\"u}fer Transform can be found in \cite{20}. 

\begin{lemma}[Pr{\"u}fer Transform]\label{380}
Suppose that the function $\phi: [a,b] \rightarrow \R$ 
satisfies the differential equation
\begin{equation}\label{400}
\phi^{\prime\prime}(x)+\alpha(x)\phi^{\prime}(x)
+\beta(x)\phi(x)=0,
\end{equation}
where $\alpha,\beta:(a,b) \rightarrow \R$ are differential functions.
Then,
\begin{equation}\label{440}
\frac{d\theta}{dx}=-\sqrt{\beta(x)}-\left(\frac{\beta^{\prime}(x)}
{4\beta(x)}+\frac{\alpha(x)}{2}\right)sin(2\theta),
\end{equation}
where the function $\theta :[a,b]\rightarrow \R$ is defined 
by the formula,
\begin{equation}\label{420}
\frac{\phi^\prime(x)}{\phi(x)}=\sqrt{\beta(x)}\tan(\theta(x)).
\end{equation}
\end{lemma}
\begin{proof}
Introducing the notation
\begin{equation}\label{460}
z(x)=\frac{\phi^{\prime}(x)}{\phi(x)}
\end{equation}
for all $x\in [a,b]$, and differentiating (\ref{460}) 
with respect to $x$, we obtain the identity
\begin{equation}\label{480}
\frac{\phi^{\prime\prime}}{\phi}=\frac{dz}{dx}+z^2(x).
\end{equation}
Substituting (\ref{480}) and (\ref{460}) into (\ref{400}),
we obtain,
\begin{equation}\label{500}
\frac{dz}{dx}=-(z^2(x)+\alpha(x)z(x)+\beta(x)).
\end{equation}
Introducing the notation,
\begin{equation}\label{520}
z(x)=\gamma(x)\tan(\theta(x)),
\end{equation}
with $\theta,\gamma$ two unknown functions, 
we differentiate (\ref{520}) and observe that,
\begin{equation}\label{540}
\frac{dz}{dx}=\gamma(x)\frac{\theta^{\prime}}{\cos^2(\theta)}
+\gamma^{\prime}(x)\tan(\theta(x))
\end{equation}
and squaring both sides of (\ref{520}), we obtain
\begin{equation}\label{560}
z(x)^2=\tan^2(\theta(x))\gamma(x)^2.
\end{equation}
Substituting (\ref{540}) and (\ref{560}) into 
(\ref{500}) and choosing 
\begin{equation}\label{580}
\gamma(x)=\sqrt{\beta(x)}
\end{equation}
we obtain
\begin{equation}\label{600}
\frac{d\theta}{dx}=-\sqrt{\beta(x)}-\left(\frac{\beta^{\prime}(x)}
{4\beta(x)}+\frac{\alpha(x)}{2}\right)sin(2\theta).
\end{equation}
\end{proof}
\begin{remark}\label{620}
The Pr{\"u}fer Transform is often used in algorithms for finding
the roots of oscillatory special functions.
Suppose that $\phi: [a,b] \rightarrow \R$ is a special 
function satisfying differential 
equation (\ref{400}). It turns out that in most cases, coefficient
\begin{equation}\label{640}
\beta(x)
\end{equation}
in (\ref{400}) is significantly larger than 
\begin{equation}\label{660}
\frac{\beta^{\prime}(x)}{4\beta(x)}+\frac{\alpha(x)}{2}
\end{equation}
on the interval $[a,b]$, where $\alpha$ and $\beta$ 
are defined in (\ref{400}).

Under these conditions, the function $\theta$ (see (\ref{420})), 
is monotone and its derivative neither approaches infinity nor $0$.
Furthermore, finding the roots of $\phi$ is equivalent to 
finding $x \in [a,b]$ such that
\begin{equation}\label{670}
\theta(x)=\pi/2+k\pi
\end{equation}
for some integer $k$.
Consequently, we can find the roots of $\phi$ by solving 
well-behaved differential equation (\ref{600}).
\end{remark}

\begin{remark}\label{680}
If for all $x\in [a,b]$, the function $\sqrt{\beta(x)}$ satisfies
\begin{equation}\label{700}
\sqrt{\beta(x)}>
\frac{\beta^{\prime}(x)}
{4\beta(x)}+\frac{\alpha(x)}{2},
\end{equation}
then, for all $x\in [a,b]$, we have 
$\frac{d\theta}{dx} < 0$ (see (\ref{440}))
and we can view $x:[-\pi,\pi]\rightarrow \R$
as a function of $\theta$ where $x$ satisfies the first order
differential equation
\begin{equation}\label{720}
\frac{dx}{d\theta}=\left(-\sqrt{\beta(x)}-\left(\frac{\beta^{\prime}(x)}
{4\beta(x)}+\frac{\alpha(x)}{2}\right)sin(2\theta)\right)^{-1}.
\end{equation}
\end{remark}
\subsection{Roots of the Shifted Jacobi Polynomials}\label{secxkquad}
The primary purpose of this section is to describe an algorithm
for finding the roots of the Jacobi polynomials $\widetilde{P}_n^k$. 
These roots will be used in Section \ref{seczernquad}
for the design of quadratures for Zernike Polynomials. \\
The following lemma follows immediately from applying the Prufer
Transform (see Lemma \ref{380}) to (\ref{800}). 
\begin{lemma}\label{900}
For all non-negative integers $n$, real $k>-1$, and $r\in(0,1)$,
\begin{equation}\label{eq1140}
\frac{d\theta}{dr}=
-\left(\frac{n(n+k+1)}{r-r^2}\right)^{1/2}
-\left(
\frac{1-2r+2k-2kr}{4(r-r^2)}
\right)
\sin(2\theta(r)).
\end{equation}
where the function $\theta:(0,1) \rightarrow \R$ is defined 
by the formula
\begin{equation}\label{eq1200}
\frac{\widetilde{P}_n^k(r)}{\widetilde{P}_n^{k\prime}(r)}=
\left(\frac{n(n+k+1)}{r-r^2}\right)^{1/2}\tan(\theta(r)),
\end{equation}
where $\widetilde{P}_n^k$ is defined in (\ref{740.2}).
\end{lemma}
\begin{remark}
For any non-negative integer $n$,
\begin{equation}
\frac{d\theta}{dr}<0
\end{equation}
for all $r\in(0,1)$. Therefore, applying Remark \ref{680} to 
(\ref{eq1140}), we can view $r$ as a function of $\theta$ where
$r$ satisfies the differential equation
\begin{equation}\label{eq1150}
\frac{dr}{d\theta}=
\left(
-\left(\frac{n(n+k+1)}{r-r^2}\right)^{1/2}
-\left(
\frac{1-2r+2k-2kr}{4(r-r^2)}
\right)
\sin(2\theta(r))
\right)^{-1}.
\end{equation}
\end{remark}
\subsubsection{Algorithm}\label{secalg}
In this section, we describe an algorithm for the evaluation
of the $n$ roots of $\widetilde{P}_n^k$. We denote 
the $n$ roots of $\widetilde{P}_n^k$ by $r_1<r_2<...<r_n$.\\\\
Step 1. Choose a point, $x_0 \in (0,1)$, that is greater than 
the largest root of $\widetilde{P}_n^k$. For example, for all $k \geq 1$, 
the following choice of $x_0$ will be sufficient:
  \begin{align}
x_0 = \left\{
  \begin{array}{ll}
  1-10^{-6} & \mbox{if $n <10^3$}, \\
  1-10^{-8} & \mbox{if $10^3 \leq n <10^4$}, \\
  1-10^{-10} & \mbox{if $10^4 \leq n <10^5$}.
  \end{array}
\right.
  \end{align}
\\\\
Step 2. Defining $\theta_0$ by the formula
\begin{equation}
\theta_0=\theta(x_0),
\end{equation}
where $\theta$ is defined in (\ref{eq1200}),
solve the ordinary differential equation $\frac{dr}{d\theta}$
(see (\ref{eq1150})) on the interval $(\pi/2,\theta_0)$, with 
the initial condition $r(\theta_0)=x_0$. 
To solve the differential equation, it is sufficient to 
use, for example, second order Runge Kutta with $100$ steps (independent of $n$).
We denote by $\tilde{r}_n$ the approximation to $r(\pi/2)$ obtained
by this process. It follows immediately from $(\ref{670})$ that
$\tilde{r}_n$ is an approximation to $r_n$, the largest root of 
$\widetilde{P}_n^k$.\\\\
Step 3. Use Newton's method with $\tilde{r}_n$ as an initial guess
to find $r_n$ to high precision. The polynomials $\widetilde{P}_n^k$ and 
$\widetilde{P}_n^{k\prime}$ can be evaluated via Observation \ref{765}
and Observation \ref{767}.\\\\
Step 4. With initial condition 
\begin{equation}
x(\pi/2)=r_n,
\end{equation}
solve differential equation $\frac{dr}{d\theta}$ (see (\ref{eq1150}))
on the interval 
\begin{equation}
(-\pi/2,\pi/2)
\end{equation}
using, for example, second order Runge Kuta with $100$ steps.
We denote by $\tilde{r}_{n-1}$ the approximation to 
\begin{equation}
r(-\pi/2)
\end{equation}
obtained by this process.\\\\
Step 5. Use Newton's method, with initial guess $\tilde{r}_{n-1}$,
to find to high precision the second largest root, $r_{n-1}$. \\\\
Step 6. For $k=\{1,2,...,n-1\}$, 
repeat Step 4 on the interval
\begin{equation}
(-\pi/2-k\pi,-\pi/2-(k-1)\pi)
\end{equation}
with intial condition 
\begin{equation}
x(-\pi/2-(k-1)\pi)=r_{n-k+1}
\end{equation}
and repeat Step 5 on $\tilde{r}_{n-k}$.

\section{Appendix E: Notational Conventions for Zernike Polynomials}
  \label{not:1}
In two dimensions, the Zernike polynomials are usually indexed by their
azimuthal order and radial order. In this report, we use a
slightly different indexing scheme, which leads to simpler formulas and 
generalizes easily to higher dimensions (see Section~\ref{seczern} for our
definition of the Zernike polynomials $Z_{N,n}^\ell$ and the radial
polynomials $R_{N,n}$). However, for the sake of completeness, we describe
in this section the standard two dimensional indexing scheme, as well as
other widely used notational conventions. 

If $|m|$ denotes the azimuthal order and $n$ the radial order, then the
Zernike polynomials in standard two index notation (using asterisks to
differentiate them from the polynomials $Z_{N,n}^\ell$ and $R_{N,n}$) are
\boxit{
  \begin{align}
\accentset{*}{Z}_n^m(\rho,\theta) = \accentset{*}{R}_n^{|m|}(\rho) \cdot
  \left\{
  \begin{array}{cc}
  \sin(|m|\theta) & \text{if $m<0$}, \\
  \cos(|m|\theta) & \text{if $m>0$}, \\
  1 & \text{if $m=0$},
  \end{array}
  \right.
    \label{not:2}
  \end{align}
}
where
  \begin{align}
\accentset{*}{R}_n^{|m|}(\rho) =
  \sum_{k=0}^{\frac{n-|m|}{2}}
  \frac{ (-1)^k (n-k)! }{ k! \bigl(\frac{n+|m|}{2} - k\bigr)!
    \bigl(\frac{n-|m|}{2} - k\bigr)! } \rho^{n-2k},
    \label{not:3}
  \end{align}
for all $m=0,\pm 1,\pm 2, \ldots$ and $n=|m|,|m|+2,|m|+4,\ldots$ 
(see Figure~\ref{not:2}); they are normalized so that
  \begin{align}
\accentset{*}{R}_n^{|m|}(1) = 1,
  \end{align}
for all $m=0,\pm 1,\pm 2, \ldots$ and $n=|m|,|m|+2,|m|+4,\ldots$ .
We note that
  \begin{align}
\accentset{*}{R}_n^{|m|}(\rho) = R_{|m|,\frac{n-|m|}{2}}(\rho),
    \label{not:4}
  \end{align}
for all $m=0,\pm 1,\pm 2, \ldots$ and $n=|m|,|m|+2,|m|+4,\ldots$,
where $R$ is defined by~(\ref{360}) (see Figure~\ref{fig:rnn_notation}); 
equivalently, 
  \begin{align}
R_{N,n}(\rho) = \accentset{*}{R}^N_{N+2n}(\rho),
    \label{not:5}
  \end{align}
for all nonnegative integers $N$ and $n$.

\begin{remark}
    \label{not:6}
The quantity $n+|m|$ is sometimes referred to as the ``spacial frequency''
of the Zernike polynomial $\accentset{*}{Z}_n^m(\rho,\theta)$. It roughly
corresponds to the frequency of the polynomial on the disc, as opposed to
the azimuthal frequency $|m|$ or the order of the polynomial $n$.

\end{remark}

\subsection{Zernike Fringe Polynomials}

The Zernike Fringe Polynomials are the standard Zernike polynomials,
normalized to have $L^2$ norm equal to $\pi$ on the unit disc and ordered by
their spacial frequency $n+|m|$ (see Table~\ref{tab:fringe} and
Figure~\ref{fig:fringe_ord}). This ordering is sometimes called the ``Air
Force'' or ``University of Arizona'' ordering.

\subsection{ANSI Standard Zernike Polynomials}

The ANSI Standard Zernike polynomials, also referred to as OSA Standard
Zernike polynomials or Noll Zernike polynomials, are the standard Zernike
polynomials, normalized to have $L^2$ norm $\pi$ on the unit disc and
ordered by $n$ (the order of the polynomial on the disc; see
Table~\ref{tab:ansi} and Figure~\ref{fig:ansi_ord}).

\subsection{Wyant and Creath Notation}
\boxit{
In~\cite{wyant}, James Wyant and Katherine Creath observe that it is
sometimes convienient to factor the radial polynomial
$\accentset{*}{R}_{2n-|m|}^{|m|}$ into
  \begin{align}
\accentset{*}{R}_{2n-|m|}^{|m|}(\rho) = Q_n^{|m|}(\rho) \rho^{|m|},
  \end{align}
for all $m=0,\pm 1,\pm 2, \ldots$ and $n=|m|,|m|+1,|m|+2,\ldots$, where the
polynomial $Q_n^{|m|}$ is of order $2(n-|m|)$ (see Figure~\ref{fig:wyant}).
Equivalently, the factorization can be written as
  \begin{align}
\accentset{*}{R}_{n}^{|m|}(\rho) = Q_{\frac{n+|m|}{2}}^{|m|}(\rho) \rho^{|m|},
  \end{align}
for all $m=0,\pm 1,\pm 2, \ldots$ and $n=|m|,|m|+2,|m|+4,\ldots$ . 
}

%
%
\begin{figure}[htp]
\centering
\begin{tikzpicture}[scale=0.6]

\tikzmath{\mmax=6; \nmax=\mmax;}

\draw (-1,-\mmax) -- (-1,\mmax);
\draw (0,-\mmax-1) -- (\nmax,-\mmax-1);

\foreach \y in {-\mmax,...,\mmax}
  \draw (-0.1-1,\y) -- (0.1-1,\y);

\foreach \y in {-\mmax,...,-1}
  \node at (-1.67,\y) {$\y$};
\foreach \y in {0,...,\mmax}
  \node at (-1.45,\y) {$\y$};

\foreach \x in {0,...,\nmax}{
  \draw (\x,-\mmax-1-0.1) -- (\x,-\mmax-1+0.1);
  \node at (\x,-\mmax-1-0.5) {$\x$};
}

\node at (-2.2,0.1) {$m$};
\node at (2.7,-\mmax-2.1) {$n$};

\tikzmath{
  int \n, \m, \ii, \jj;
  \n=0;
  for \n in {0,...,\nmax}%
  {
    for \m in {-\n,...,\n}%
    {
      if mod(abs(\m)-\n,2) == 0 then {
        \ii=abs(\m);
        \jj=(\n+abs(\m))/2;
        { \draw (\n,\m) circle [radius=0.57];
          \node at (\n,\m) {\footnotesize $Q_\ii^{\jj}$}; };
      } else {
        { \draw (\n-0.1,\m) -- (\n+0.1,\m);
          \draw (\n,\m-0.1) -- (\n,\m+0.1); };
      };
    };
  };
}

\end{tikzpicture}

\caption{ The Wyant and Creath polynomials $Q^{|m|}_{(n+|m|)/2}$ }
\label{fig:wyant}
\end{figure}
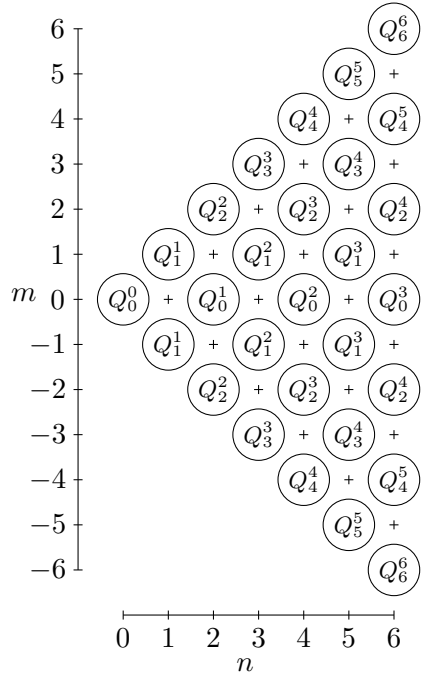

%
%
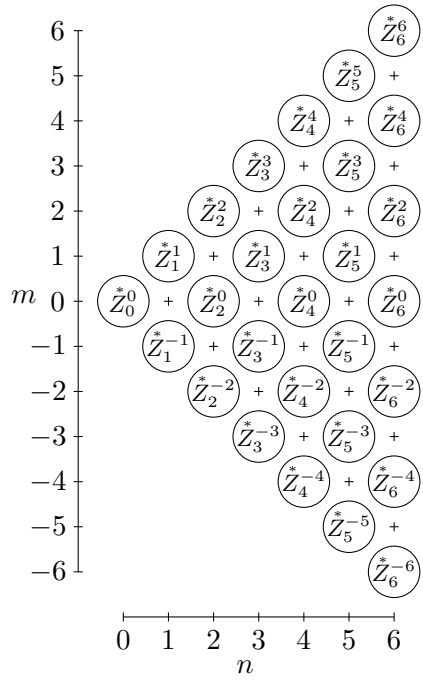
\begin{figure}[htp]
\centering
\begin{tikzpicture}[scale=0.6]

\tikzmath{\mmax=6; \nmax=\mmax;}

\draw (-1,-\mmax) -- (-1,\mmax);
\draw (0,-\mmax-1) -- (\nmax,-\mmax-1);

\foreach \y in {-\mmax,...,\mmax}
  \draw (-0.1-1,\y) -- (0.1-1,\y);

\foreach \y in {-\mmax,...,-1}
  \node at (-1.67,\y) {$\y$};
\foreach \y in {0,...,\mmax}
  \node at (-1.45,\y) {$\y$};

\foreach \x in {0,...,\nmax}{
  \draw (\x,-\mmax-1-0.1) -- (\x,-\mmax-1+0.1);
  \node at (\x,-\mmax-1-0.5) {$\x$};
}

\node at (-2.2,0.1) {$m$};
\node at (2.7,-\mmax-2.1) {$n$};

\tikzmath{
  int \n, \m;
  \n=0;
  for \n in {0,...,\nmax}%
  {
    for \m in {-\n,...,\n}%
    {
      if mod(abs(\m)-\n,2) == 0 then {
        { \draw (\n,\m) circle [radius=0.57];
          \node at (\n,\m) {\footnotesize $\protect
          \accentset{*}{Z}_\n^{\m}$}; };
      } else {
        { \draw (\n-0.1,\m) -- (\n+0.1,\m);
          \draw (\n,\m-0.1) -- (\n,\m+0.1); };
      };
    };
  };
}

\end{tikzpicture}

\caption{ { Zernike polynomials $\protect\accentset{*}{Z}^m_n$ in standard
double index notation} }
\label{fig:standard}
\end{figure}

%
%
\begin{figure}[htp]
\centering
\begin{tikzpicture}[scale=0.6]

\tikzmath{\mmax=6; \nmax=\mmax;}

\draw (-1,-\mmax) -- (-1,\mmax);
\draw (0,-\mmax-1) -- (\nmax,-\mmax-1);

\foreach \y in {-\mmax,...,\mmax}
  \draw (-0.1-1,\y) -- (0.1-1,\y);

\foreach \y in {-\mmax,...,-1}
  \node at (-1.67,\y) {$\y$};
\foreach \y in {0,...,\mmax}
  \node at (-1.45,\y) {$\y$};

\foreach \x in {0,...,\nmax}{
  \draw (\x,-\mmax-1-0.1) -- (\x,-\mmax-1+0.1);
  \node at (\x,-\mmax-1-0.5) {$\x$};
}

\node at (-2.2,0.1) {$m$};
\node at (2.7,-\mmax-2.1) {$n$};

\tikzmath{
  int \n, \m, \N, \nn;
  \n=0;
  for \n in {0,...,\nmax}%
  {
    for \m in {-\n,...,\n}%
    {
      if mod(abs(\m)-\n,2) == 0 then {
        \N=abs(\m);
        \nn=(\n-abs(\m))/2;
        { \draw (\n,\m) circle [radius=0.57];
          \node at (\n,\m) {\footnotesize $R_{\N,\nn}$}; };
      } else {
        { \draw (\n-0.1,\m) -- (\n+0.1,\m);
          \draw (\n,\m-0.1) -- (\n,\m+0.1); };
      };
    };
  };
}

\end{tikzpicture}

\caption{ The polynomials $R_{|m|,(n-|m|)/2}$ }
\label{fig:rnn_notation}
\end{figure}
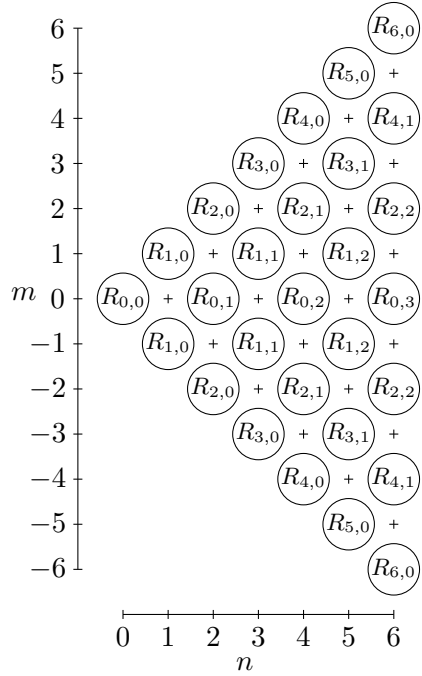

%
%
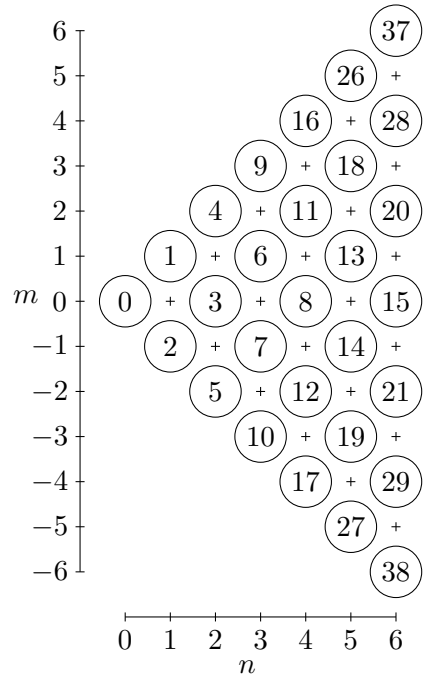
\begin{figure}[htp]
\centering
\begin{tikzpicture}[scale=0.6]

\tikzmath{\mmax=6; \nmax=\mmax;}

\draw (-1,-\mmax) -- (-1,\mmax);
\draw (0,-\mmax-1) -- (\nmax,-\mmax-1);

\foreach \y in {-\mmax,...,\mmax}
  \draw (-0.1-1,\y) -- (0.1-1,\y);

\foreach \y in {-\mmax,...,-1}
  \node at (-1.67,\y) {$\y$};
\foreach \y in {0,...,\mmax}
  \node at (-1.45,\y) {$\y$};

\foreach \x in {0,...,\nmax}{
  \draw (\x,-\mmax-1-0.1) -- (\x,-\mmax-1+0.1);
  \node at (\x,-\mmax-1-0.5) {$\x$};
}

\node at (-2.2,0.1) {$m$};
\node at (2.7,-\mmax-2.1) {$n$};

\tikzmath{
  int \n, \m, \l;
  \n=0;
  for \n in {0,...,\nmax}%
  {
    for \m in {-\n,...,\n}%
    {
      if mod(abs(\m)-\n,2) == 0 then {
        \l=-7;
        if \n == 0 && \m == 0 then {\l=0;};
        if \n == 1 && \m == 1 then {\l=1;};
        if \n == 1 && \m == -1 then {\l=2;};
        if \n == 2 && \m == 0 then {\l=3;};
        if \n == 2 && \m == 2 then {\l=4;};
        if \n == 2 && \m == -2 then {\l=5;};
        if \n == 3 && \m == 1 then {\l=6;};
        if \n == 3 && \m == -1 then {\l=7;};
        if \n == 4 && \m == 0 then {\l=8;};
        if \n == 3 && \m == 3 then {\l=9;};
        if \n == 3 && \m == -3 then {\l=10;};
        if \n == 4 && \m == 2 then {\l=11;};
        if \n == 4 && \m == -2 then {\l=12;};
        if \n == 5 && \m == 1 then {\l=13;};
        if \n == 5 && \m == -1 then {\l=14;};
        if \n == 6 && \m == 0 then {\l=15;};
        if \n == 4 && \m == 4 then {\l=16;};
        if \n == 4 && \m == -4 then {\l=17;};
        if \n == 5 && \m == 3 then {\l=18;};
        if \n == 5 && \m == -3 then {\l=19;};
        if \n == 6 && \m == 2 then {\l=20;};
        if \n == 6 && \m == -2 then {\l=21;};
        if \n == 5 && \m == 5 then {\l=26;};
        if \n == 5 && \m == -5 then {\l=27;};
        if \n == 6 && \m == 4 then {\l=28;};
        if \n == 6 && \m == -4 then {\l=29;};
        if \n == 6 && \m == 6 then {\l=37;};
        if \n == 6 && \m == -6 then {\l=38;};
        { \draw (\n,\m) circle [radius=0.57];
          \node at (\n,\m) {$\l$}; };
      } else {
        { \draw (\n-0.1,\m) -- (\n+0.1,\m);
          \draw (\n,\m-0.1) -- (\n,\m+0.1); };
      };
    };
  };
}

\end{tikzpicture}

\caption{ Fringe Zernike Polynomial Ordering }
  \label{fig:fringe_ord}
\end{figure}

%
%
\begin{figure}[htp]
\centering
\begin{tikzpicture}[scale=0.6]

\tikzmath{\mmax=6; \nmax=\mmax;}

\draw (-1,-\mmax) -- (-1,\mmax);
\draw (0,-\mmax-1) -- (\nmax,-\mmax-1);

\foreach \y in {-\mmax,...,\mmax}
  \draw (-0.1-1,\y) -- (0.1-1,\y);

\foreach \y in {-\mmax,...,-1}
  \node at (-1.67,\y) {$\y$};
\foreach \y in {0,...,\mmax}
  \node at (-1.45,\y) {$\y$};

\foreach \x in {0,...,\nmax}{
  \draw (\x,-\mmax-1-0.1) -- (\x,-\mmax-1+0.1);
  \node at (\x,-\mmax-1-0.5) {$\x$};
}

\node at (-2.2,0.1) {$m$};
\node at (2.7,-\mmax-2.1) {$n$};

\tikzmath{
  int \n, \m, \l;
  \n=0;
  for \n in {0,...,\nmax}%
  {
    for \m in {-\n,...,\n}%
    {
      if mod(abs(\m)-\n,2) == 0 then {
        \l=( (\n*(\n+2))+ \m )/2;
        { \draw (\n,\m) circle [radius=0.57];
          \node at (\n,\m) {$\l$}; };
      } else {
        { \draw (\n-0.1,\m) -- (\n+0.1,\m);
          \draw (\n,\m-0.1) -- (\n,\m+0.1); };
      };
    };
  };
}

\end{tikzpicture}

\caption{ ANSI Standard Zernike Polynomial Ordering }
  \label{fig:ansi_ord}
\end{figure}
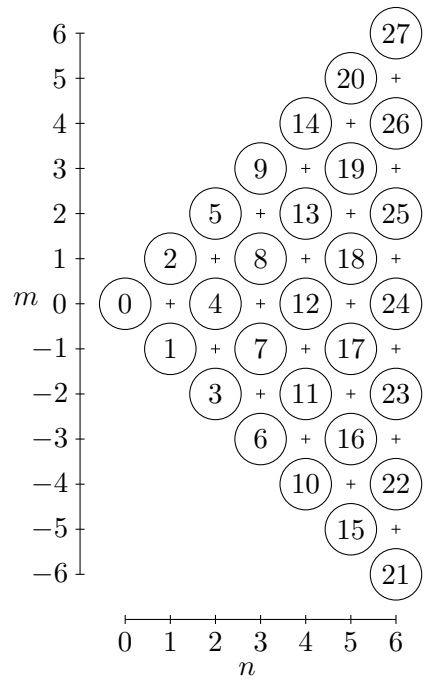

\begin{table}[htp]

  \centering
\begin{tabu}{llllll}
\hline
\multirow{2}{*}{index} & 
\multirow{2}{*}{$n$} & 
\multirow{2}{*}{$m$} & 
\multirow{2}{*}{ \parbox{5em}{ spacial \\ frequency } }& 
\multirow{2}{*}{polynomial$^\diamond$} & 
\multirow{2}{*}{name$^\dagger$} \\ 
 & & & & & \\
\hline
0 & 0 & 0 & 0 & $1$ & 
  piston \\
\tabucline[mydashed]{-}
1 & 1 & 1 & 2 & $\overbar R_{1,0}(\rho) \cos(\theta)$ & 
  tilt in $x$-direction \\
2 & 1 & -1 & 2 & $\overbar R_{1,0}(\rho) \sin(\theta)$ &
  tilt in $y$-direction \\
3 & 2 & 0 & 2 & $\overbar R_{0,1}(\rho)$ & 
  defocus (power) \\
\tabucline[mydashed]{-}
4 & 2 & 2 & 4 & $\overbar R_{2,0}(\rho) \cos(2\theta)$ &
  defocus + astigmatism 45$^\circ$/135$^\circ$ \\
5 & 2 & -2 & 4 & $\overbar R_{2,0}(\rho) \sin(2\theta)$ & 
  defocus + astigmatism 90$^\circ$/180$^\circ$ \\
6 & 3 & 1 & 4 & $\overbar R_{1,1}(\rho) \cos(\theta)$ &
  tilt + horiz. coma along $x$-axis \\
7 & 3 & -1 & 4 & $\overbar R_{1,1}(\rho) \sin(\theta)$ & 
  tilt + vert. coma along $y$-axis \\
8 & 4 & 0 & 4 & $\overbar R_{0,2}(\rho)$ &
  defocus + spherical aberration \\
\tabucline[mydashed]{-}
9  & 3 & 3 & 6 & $\overbar R_{3,0}(\rho) \cos(3\theta)$ &
  trefoil in $x$-direction \\
10 & 3 & -3 & 6 & $\overbar R_{3,0}(\rho) \sin(3\theta)$ &
  trefoil in $y$-direction \\
11 & 4 & 2 & 6 & $\overbar R_{2,1}(\rho) \cos(2\theta)$ &  \\
12 & 4 & -2 & 6 & $\overbar R_{2,1}(\rho) \sin(2\theta)$ &  \\
13 & 5 & 1 & 6 & $\overbar R_{1,2}(\rho) \cos(\theta)$ &  \\
14 & 5 & -1 & 6 & $\overbar R_{1,2}(\rho) \sin(\theta)$ &  \\
15 & 6 & 0 & 6 & $\overbar R_{0,3}(\rho)$ & \\
\tabucline[mydashed]{-}
16  & 4 & 4 & 8 & $\overbar R_{4,0}(\rho) \cos(4\theta)$ & \\
17 & 4 & -4 & 8 & $\overbar R_{4,0}(\rho) \sin(4\theta)$ & \\
18 & 5 & 3 & 8 & $\overbar R_{3,1}(\rho) \cos(3\theta)$ &  \\
19 & 5 & -3 & 8 & $\overbar R_{3,1}(\rho) \sin(3\theta)$ &  \\
20 & 6 & 2 & 8 & $\overbar R_{2,2}(\rho) \cos(2\theta)$ &  \\
21 & 6 & -2 & 8 & $\overbar R_{2,2}(\rho) \sin(2\theta)$ &  \\
22 & 7 & 1 & 8 & $\overbar R_{1,3}(\rho) \cos(\theta)$ & \\
23 & 7 & -1 & 8 & $\overbar R_{1,3}(\rho) \sin(\theta)$ &  \\
24 & 8 & 0 & 8 & $\overbar R_{0,4}(\rho)$ & \\
\hline

\end{tabu}

\caption{ {\bf Zernike Fringe Polynomials.} This table lists the first 24
Zernike polynomials in what is sometimes called the ``Fringe'', ``Air
Force'', or ``University of Arizona'' ordering (see, for
example~\cite{zemax}, p.~198, or~\cite{wyant}, p.~31).  They are often also
denoted by $Z_\ell(\rho,\theta)$, where $\ell$ is the index. \\
$\diamond$~See formulas~(\ref{360}) and~(\ref{1.65}). \\
$\dagger$~See, for example,~\cite{jagerman}. More complex aberrations
are usually not named. }
  \label{tab:fringe}
\end{table}

\begin{table}[htp]

  \centering
\begin{tabu}{llllll}
\hline
\multirow{2}{*}{index} & 
\multirow{2}{*}{$n$} & 
\multirow{2}{*}{$m$} & 
\multirow{2}{*}{ \parbox{5em}{spacial \\ frequency} }& 
\multirow{2}{*}{polynomial$^\diamond$} & 
\multirow{2}{*}{name$^\dagger$} \\ 
 & & & & & \\
\hline
0 & 0 & 0 & 0 & $1$ & 
  piston \\
\tabucline[mydashed]{-}
1 & 1 & -1 & 2 & $\overbar R_{1,0}(\rho) \sin(\theta)$   & 
  tilt in $y$-direction \\
2 & 1 & 1 & 2 & $\overbar R_{1,0}(\rho) \cos(\theta)$    & 
  tilt in $x$-direction \\
\tabucline[mydashed]{-}
3 & 2 & -2 & 4 & $\overbar R_{2,0}(\rho) \sin(2\theta)$  & 
  defocus + astigmatism 90$^\circ$/180$^\circ$ \\
4 & 2 & 0 & 2 & $\overbar R_{0,1}(\rho)$                 &
  defocus (power)  \\
5 & 2 & 2 & 4 & $\overbar R_{2,0}(\rho) \cos(2\theta)$   &
  defocus + astigmatism 45$^\circ$/135$^\circ$ \\
\tabucline[mydashed]{-}
6 & 3 & -3 & 6 & $\overbar R_{3,0}(\rho) \sin(3\theta)$  &
  trefoil in $y$-direction \\
7 & 3 & -1 & 4 & $\overbar R_{1,1}(\rho) \sin(\theta)$   &
  tilt + vert. coma along $y$-axis \\
8 & 3 & 1 & 4 & $\overbar R_{1,1}(\rho)  \cos(\theta)$   & 
  tilt + horiz. coma along $x$-axis \\
9  & 3 & 3 & 6 & $\overbar R_{3,0}(\rho) \cos(3\theta)$  &
  trefoil in $x$-direction \\
\tabucline[mydashed]{-}
10 & 4 & -4 & 8 & $\overbar R_{4,0}(\rho) \sin(4\theta)$ & \\
11 & 4 & -2 & 6 & $\overbar R_{2,1}(\rho) \sin(2\theta)$ & \\
12 & 4 &  0 & 4 & $\overbar R_{0,2}(\rho)$               & \\
13 & 4 &  2 & 6 & $\overbar R_{2,1}(\rho) \cos(2\theta)$ & \\
14 & 4 &  4 & 8 & $\overbar R_{4,0}(\rho) \cos(4\theta)$ & \\
\tabucline[mydashed]{-}
15 & 5 & -5 & 10 & $\overbar R_{5,0}(\rho) \sin(5\theta)$ & \\
16 & 5 & -3 & 8 & $\overbar R_{3,1}(\rho) \sin(3\theta)$ & \\
17 & 5 & -1 & 6 & $\overbar R_{1,2}(\rho) \sin(\theta)$ & \\
18 & 5 &  1 & 6 & $\overbar R_{1,2}(\rho) \cos(\theta)$ & \\
19 & 5 &  3 & 8 & $\overbar R_{3,1}(\rho) \cos(3\theta)$ & \\
20 & 5 &  5 & 10 & $\overbar R_{5,0}(\rho) \cos(5\theta)$ & \\
\tabucline[mydashed]{-}
21 & 6 & -6 & 12 & $\overbar R_{6,0}(\rho) \sin(6\theta)$ & \\
22 & 6 & -4 & 10 & $\overbar R_{4,1}(\rho) \sin(4\theta)$ & \\
23 & 6 & -2 & 8 & $\overbar R_{2,2}(\rho) \sin(2\theta)$ & \\
24 & 6 &  0 & 6 & $\overbar R_{0,3}(\rho)$ & \\
\hline

\end{tabu}

\caption{ {\bf ANSI Standard Zernike Polynomials.} This table lists the
first 24 Zernike polynomials in the ANSI Standard ordering, also referred to
as the ``OSA Standard'' or ``Noll'' ordering (see, for example~\cite{zemax},
p.~201, or~\cite{ansi}).  They are often also denoted by
$Z_\ell(\rho,\theta)$, where $\ell$ is the index. \\ 
$\diamond$~See formulas~(\ref{360}) and~(\ref{1.65}). \\
$\dagger$~See, for example,~\cite{jagerman}. More complex aberrations are
usually not named. }
  \label{tab:ansi}
\end{table}

\end{document}